\documentclass[12pt,letter,reqno]{amsart}

\usepackage{graphicx}

\usepackage{amssymb}
\usepackage{relsize}

\usepackage{amsmath,amsthm}
\usepackage{mathtools}

\usepackage{tabularx}
\usepackage{subfig}
\usepackage{color,soul,colortbl} 
\usepackage{sidecap}

\usepackage{tikz}
\usetikzlibrary{fit}
\usetikzlibrary{shapes}
\usetikzlibrary{arrows}
\usetikzlibrary{decorations.markings}
\usetikzlibrary{positioning}

\tikzset{mylabel/.style={font=\footnotesize}}
\tikzset{mymidlabel/.style={fill=white}}

\definecolor{mydark}{RGB}{73,68,62}
\definecolor{mymedium}{RGB}{128,129,135}
\definecolor{mylight}{RGB}{226,226,226}
\definecolor{myyellow}{RGB}{251,214,0}
\definecolor{mydarkyellow}{RGB}{255,204,13}

\def\R{\mathbb{R}}
\def\Sbb{\mathbb{S}}
\def\e{\mathrm{e}}
\def\calF{\mathcal{F}}

\theoremstyle{plain}
\newtheorem{thm}{Theorem}[section]
\newtheorem{lem}{Lemma}[section]

\newtheorem*{thm*}{Theorem}
\newtheorem*{lem*}{Lemma}

\theoremstyle{remark}

\definecolor{gray80}{gray}{0.80}
\definecolor{gray90}{gray}{0.90}
\definecolor{gray95}{gray}{0.95}
\sethlcolor{lightgray}

\begin{document}

\title[]{Evaluating High Order Discontinuous Galerkin 
Discretization of the Boltzmann Collision Integral in $\mathcal{O}(N^2)$ Operations Using the Discrete Fourier Transform}

\author[A.\ Alekseenko]{Alexander~Alekseenko}
\address{Department of Mathematics, California State University Nortrhidge, Nortrhdige, CA 91330, USA}
\email{alexander.alekseenko@csun.edu}
\thanks{The first author was supported by the NSF DMS-1620497}

\author[J.\ Limbacher]{Jeffrey Limbacher}
\email{jeffrey.limbacher.248@my.csun.edu}
\thanks{The second author was supported by the NSF DMS-1620497}

\thanks{Computer resources were provided by the Extreme Science and Engineering Discovery Environment, 
supported by National Science Foundation Grant No.~OCI-1053575.}



\begin{abstract}
We present a numerical algorithm for evaluating the Boltzmann collision operator
with $O(N^2)$ operations based on high order discontinuous Galerkin 
discretizations in 
the velocity variable. To formulate the approach, Galerkin projection 
of the collision operator is written in the form of a bilinear 
circular convolution. An application of the discrete Fourier transform allows to 
rewrite
the six fold convolution sum as a three fold weighted convolution sum in the 
frequency space. The new algorithm is implemented and tested in the 
spatially homogeneous case, and results in a considerable improvement 
in speed as compared to the direct evaluation. Simultaneous and separate 
evaluations of the gain and loss terms of the collision operator were 
considered. Less numerical error was observed in the conserved quantities  with  
simultaneous evaluation. 
\end{abstract}

\maketitle

\section{Introduction}
\label{secN1}

It has been accepted for some time that using global Fourier basis functions in 
velocity discretizations of the Boltzmann equation is essential in order to achieve $O(N^2)$ 
evaluation of the collision integral, where $N$ is the total number of discretization points. 
Exponentials have the factorization 
property that can be used to transform the gain term of the collision operator 
into forms suitable for efficient computation 
\cite{PareschiPerthame1996,KirschRjasanov2007,GambaTharkabhushanam2010}.
Another essential attribute of an efficient numerical 
formulation of the  collision operator consists in re-writing it  
in the form of a convolution \cite{BOBYLEV1999,BOBYLEVRjasanov1997,
MouhotPareschi2006,FilbetMouhotPareschi2006,FilbetMouhot2011,HuYing2012,HuLiPareschi2015}. 
A bilinear convolution form \cite{AlekseenkoNguyenWood2015}  follows for the Galerkin projection of the collision operator by exploring translational invariance of the collision operator \cite{GHP17_2133}. We argue in this paper that this convolution 
form leads to development of efficient discretizations of the collision operator
using structured locally supported bases. 
We present a numerical approach that is based on high 
order nodal discontinuous Galerkin (DG) discretizations of the Boltzmann equation in the velocity variable \cite{AlekseenkoJosyula2012a}
and that requires $O(N^2)$ operations to evaluate the collision operator.

Deterministic solution of the Boltzmann equation has a rich history. Readers interested in recent developments are 
directed to the review articles \cite{dimarco_pareschi_2014,NarayanKlochner}. A review of earlier results can be found in \cite{Aristov2001}.   
The difficulty in solving the Boltzmann equation is the evaluation of the five-fold collision integral describing interactions
of the gas molecules. Methods have been proposed to solve the Boltzmann equation using the direct 
discretization of the collision operator in the velocity variable (see, e.g., 
\cite{AristovZabelok2002,PanferovHeintz2002,Babovsky2008}). 
Computational costs of the direct methods grow very rapidly, usually, at least as $O(n^8)$, where $n$ is the 
number of velocity points in one velocity dimension. As a result, these methods can only 
be applied to problems that do not require a large number of spatial discretization points.

A number of approaches to solve the Boltzmann equation were obtained by applying spectral discretizations 
in the velocity variable. In \cite{PareschiPerthame1996} a Galerkin discretization was constructed using the Fourier 
basis functions. By exploring properties of exponentials, the authors formulated an approach 
for evaluation 
of the collision operator with $O(k^6)$ operations, where $k$ is the number of Fourier basis functions 
used in each velocity dimension. A closely related approach based on an application of the Fourier transform to the collision integral 
can be found in \cite{KirschRjasanov2007,GambaTharkabhushanam2009,GambaTharkabhushanam2010}. The method has computational 
complexity $O(k^6)$ and, similarly to Galerkin spectral methods, is derived using properties 
of exponentials. In \cite{MouhotPareschi2006,FilbetMouhotPareschi2006,FilbetMouhot2011} Carleman
representation of the collision operator was used to derive an approach based on Fourier-Galerkin 
discretization of the kinetic solution that uses $O(m k^3\log k)$ operations. Here $m$ is related to 
the number of angular directions in discretizations of the collision integral using spherical coordinates. An $O(m k^3\log k)$ algorithm was recently proposed for Fourier transform based 
formulations in \cite{GambaHaakHauckHu2017}. Fast spectral methods were applied to multidimensional solution of the Boltzmann equation \cite{Wu201327}, flows of gas mixtures \cite{Wu2015602}, and flows of gas with internal 
energies \cite{Munafo2014152}. A drawback of the methods is lack of adaptivity in the 
velocity space since methods use global Fourier basis. 

Alternative approaches to Fourier spectral discretization of the Boltzmann equation have been pursued 
as well. In \cite{FonnGrohsHipmair2015} a hyperbolic cross approximation of the solution in 
the frequency space was proposed introducing adaptivity in spectral methods. However, 
incomplete spectral representations are hard to combine with the use of fast Fourier 
transforms which may reduce the method's speed. A polynomial spectral discretization was  
proposed in \cite{GHP17_2133} 
and applied to solution of two dimensional super sonic flows. While the approach offers more 
compact approximations of the kinetic solution than the Fourier spectral approach, its  
algebraic complexity is higher due to the form of the collision operator used in the discrete 
algorithm. 

Approaches based on DG discretizations of the Boltzmann 
equation in the velocity variable were proposed in
\cite{Majorana2011,AlekseenkoJosyula2012,AlekseenkoJosyula2012a,GambaZhang2014}. 
High order DG bases are well suited for approximating discontinuous and 
high gradient solutions. In this paper, we present an algorithm for computing 
the collision operator in DG velocity formulations in $O(n^6)$ operations. 

Our approach is based on re-writing the discretized collision operator in the form of a discrete
convolution. It has been noted in \cite{GHP17_2133} that the collision operator 
satisfies translational 
invariance. In \cite{AlekseenkoJosyula2012,AlekseenkoJosyula2012a,GambaZhang2014} a closely 
related translational invariance of the Galerkin projection of the collision operator was 
used to reduce the storage requirements
of pre-computed collision kernels. In \cite{AlekseenkoNguyenWood2015}, the translational invariance 
was used to introduce a bilinear convolution 
form of the Galerkin projection of the collision operator. 
We will show that, in the case of uniform grids, this convolution form 
allows to re-write the collision 
operator as a convolution of multidimensional sequences. In fact, this is exactly the form of 
the collision operator that was used in \cite{AlekseenkoJosyula2012,AlekseenkoJosyula2012a} with 
the convolution being computed directly in $O(n^8)$ arithmetic operations. In this paper, the 
discrete convolution is evaluated using the discrete Fourier transform in 
only $O(n^6)$ operations. We note that the discrete Fourier transform can be replaced with a 
suitable number theoretical transform \cite{Nussbaumer1982}. Thus, one 
could, in principle, avoid using 
complex exponentials altogether. However using the discrete Fourier transform 
is convenient. The presented approach is easy to parallelize to a large number of  
processors. Evaluation of collision operator can be done in a fraction of time compared to 
the original DG velocity method \cite{AlekseenkoJosyula2012,AlekseenkoJosyula2012a}. 
Generalizations of the integral convolution form to octree meshes can be proposed. 
Difficulties arise, however,
with extending discrete convolution to non-uniform grids. In the case of 
piece-wise constant DG approximations, the new method has very similar properties to the 
Fourier-Galerkin approaches \cite{KirschRjasanov2007,GambaTharkabhushanam2010,PareschiPerthame1996}. 
One advantage of the new method is related to the use of high order DG approximations. 
Sizes of discrete velocity meshes that one can use with Fourier spectral discretizations 
are limited due to the memory requirements to compute Fourier transforms of the six 
dimensional collision kernels. 
The size of the discrete convolution in the new method is determined by the number of uniform 
velocity cells. High order nodal-DG bases can be used to approximate the solution 
accurately using $s^3$ nodal points/basis functions inside each velocity cell while 
keeping the number of cells relatively small. 
However, this introduces the total of $s^9$ convolutions that need to be evaluated. Thus, there is a practical limitation to the highest order of the DG approximation that one could employ in simulations. However, cases of $s\le 5$ are 
practically conceivable since the computations are easy to parallelize. 

In this paper, we focus on formulating the method and on establishing a 
comparison to the original approach of \cite{AlekseenkoJosyula2012,AlekseenkoJosyula2012a}. Application of the method to solution of 
spatially inhomogeneous Boltzmann equation will be the subject of the author's future work. The paper 
is organized as follows. Section~\ref{secN2} is dedicated to preliminaries on the Boltzmann equation 
and the nodal-DG 
discretizations. In Section~\ref{secN3}, the discrete convolution form of the collision operator is 
introduced and the discrete Fourier transform is used to rewrite the collision operator in the form
suitable for computation in $O(n^6)$ operations. The computational algorithms is formulated in Section~\ref{secN4}. 
In Section~\ref{secN5} we present results of numerical evaluations and comparison to the original
method of \cite{AlekseenkoJosyula2012a}. We 
compare numerical properties of the collision operator in the cases when gain and loss terms are 
computed simultaneously and separately. 

\section{The Nodal-DG Velocity Discretization}
\label{secN2}

\subsection{The Boltzmann equation}
\label{secN21}
In kinetic approach the gas is described using the molecular velocity distribution function
$f(t,\vec{x},\vec{v})$ which is defined by the following property: $f(t,\vec{x},\vec{v})d\vec{x}\,d\vec{v}$ 
gives the number of molecules that are contained in the box with the volume $d\vec{x}$ around point $\vec{x}$
whose velocities are contained in a box of volume $d\vec{v}$ around point $\vec{v}$.
Here by $d\vec{x}$ and $d\vec{v}$ we denote the volume elements $dx\, dy\, dz$ and $du\, dv\, dw$, correspondingly.
Evolution of the molecular distribution function is governed by the 
Boltzmann equation, which, in the case of one component 
atomic gas, has the form
\begin{equation}
\label{boltzm1}
\frac{\partial}{\partial t} f(t,\vec{x},\vec{v}) + \vec{v}\cdot \nabla_{x} f(t,\vec{x},\vec{v}) = I[f](t,\vec{x},\vec{v}).
\end{equation}
Here $I[f](t,\vec{x},\vec{v})$ is the molecular collision operator. In many instances, it is sufficient to only consider 
binary collisions between molecules. In this case the collision operator takes the 
form
\begin{equation}
\label{boltzm2}
I[f](t,\vec{x},\vec{v}) = \int_{\R^3}\int_{\Sbb^2}
(f(t,\vec{x},\vec{v}')f(t,\vec{x},\vec{v}'_{1})-
f(t,\vec{x},\vec{v})f(t,\vec{x},\vec{v}_{1}))B(|g|,\cos\theta) \, d\sigma\, d \vec{v}_{1},
\end{equation}
where $\vec{v}$ and $\vec{v}_{1}$ are the pre-collision velocities of a pair of molecules, $\vec{g}=\vec{v}-\vec{v}_{1}$, 
$\Sbb^2$ is a unit sphere in $\R^{3}$ centered at the origin, $\vec{w}$ is
the unit vector connecting the origin and a point on $\Sbb^2$, $\theta$ is the deflection 
angle defined by the equation $\cos\theta = \vec{w}\cdot\vec{g}/|g|$, $d\sigma=\sin\theta\, d\theta d\varepsilon$, where 
$\varepsilon$ is the azimuthal angle that parametrizes $\vec{w}$ together with the angle $\theta$. Vectors $\vec{v}'$ 
and $\vec{v}'_{1}$ are the post-collision velocities of a pair of particles and are computed by  
\begin{equation}
\label{eq2.1.3}
\vec{v}'=\vec{v}-\frac{1}{2}(\vec{g}-|g|\vec{w}), \qquad 
\vec{v}'_{1}=\vec{v}-\frac{1}{2}(\vec{g}+|g|\vec{w})\, .
\end{equation}
The kernel $B(|g|,\cos\theta)$ characterizes interactions of the molecules and is selected appropriately 
to reproduce the desired characteristics of the gas. Various forms of $B(|g|,\cos\theta)$ exist, see e.g., 
\cite{Boyd1991411,Cercignani2000,Kogan1969}. In the case of inverse 
$k$-th power forces between particles,
\begin{equation}
\label{eq2.1.4}
B(|g|,\cos\theta)=b_{\alpha}(\theta) |g|^{\alpha}, 
\end{equation}
where $\alpha=(k-5)/(k-1)$.  
The case $\alpha=0$ is known as Maxwellian gas and the case $\alpha=1$ as the hard spheres gas.
In this paper we consider kernels of the form (\ref{eq2.1.4}) with $0\le \alpha\le 1$, and with angular cut off, i.e., 
\begin{equation*}
\int_{0}^{\pi} b_{\alpha} (\theta) \sin \theta \, d\theta \, < \infty . 
\end{equation*}

\subsection{Discontinuous Galerkin velocity discretization}
\label{secN22}

The nodal-DG velocity discretization that will be employed in this paper was also used in
\cite{AlekseenkoJosyula2012,AlekseenkoJosyula2012a}. We select a rectangular parallelepiped in the 
velocity space that is sufficiently large so that contributions of the molecular distribution 
function to the first few moments outside of this parallelepiped are negligible. 
We partition this region into parallelepipeds $K_{j}$. 
Let $\vec{v}=(u,v,w)$ and let the numbers $s_u$, $s_{v}$, and $s_{w}$ determine the degrees of the polynomial 
basis functions in the velocity components $u$, $v$, and $w$, respectively. Let
$K_{j}=[u_{j}^{L},u_{j}^{R}]\times[v_{j}^{L},v_{j}^{R}]\times[w_{j}^{L},w_{j}^{R}]$.
The basis functions are constructed as follows. We introduce nodes of the Gauss
quadratures of orders $s_u$, $s_{v}$, and $s_{w}$ on each of the intervals $[u_{j}^{L},u_{j}^{R}]$,
$[v_{j}^{L},v_{j}^{R}]$, and $[w_{j}^{L},w_{j}^{R}]$, respectively.  Let these nodes be 
denoted 
$\kappa^{u}_{p;j}$, $p=1,\ldots,s_{u}$,
$\kappa^{v}_{q;j}$, $q=1,\ldots,s_{v}$, and 
$\kappa^{w}_{r;j}$, $r=1,\ldots,s_{w}$. 
We define one-dimensional Lagrange basis functions as follows, 
\begin{equation*}
\phi^{u}_{l;j}(u)=\prod_{{p=1,s^{u}} \atop {p\neq l}} \frac{\kappa^{u}_{p;j}-u}{\kappa^{u}_{p;j}-\kappa^{u}_{l;j}}\, ,\quad 
\phi^{v}_{m;j}(v)=\prod_{{q=1,s^{v}} \atop {q\neq m}} \frac{\kappa^{v}_{q;j}-v}{\kappa^{v}_{q;j}-\kappa^{v}_{m;j}}\, ,\quad 
\phi^{w}_{n;j}(w)=\prod_{{r=1,s^{w}} \atop {r\neq n}} \frac{\kappa^{w}_{r;j}-w}{\kappa^{w}_{r;j}-\kappa^{w}_{n;j}}\, .
\end{equation*}
The three-dimensional basis functions are given by
\begin{equation}
\label{eq01a}
\phi_{i;j}(\vec{v})=\phi^{u}_{l;j}(u)\phi^{v}_{m;j}(v)
\phi^{w}_{n;j}(w)\, ,
\end{equation} 
where
$i=1,\ldots,s:=s_{u}s_{v}s_{w}$ is the index running through all 
combinations of $l$, $n$, and $m$. In the implementation discussed in this paper, 
$i$ is computed using the formula $i=(l-1)s_{v}s_{w}+(m-1)s_{w}+n$. 
\begin{lem} (see also \cite{AlekseenkoJosyula2012a, HesthavenWarburtoin2007})
The following identities hold for basis functions $\phi_{i;j}(\vec{v})$:
\begin{equation}
\label{eq1.1} 
\int_{K_{j}} \phi_{p;j}(\vec{v})\phi_{q;j}(\vec{v})\, d\vec{v} = \frac{\omega_{p}\Delta\vec{v}^{j}}{8}\delta_{pq}
\qquad\mbox{and} \qquad
\int_{K_{j}} \vec{v}\phi_{p;j}(\vec{v})\phi_{q;j}(\vec{v})\, d\vec{v} 
= \frac{\omega_{p}\Delta\vec{v}^{j}}{8}\vec{v}_{p;j}\delta_{pq}\, ,
\end{equation}
where indices $p$ and $q$ run over all combinations of $l$, $n$, and $m$ in three dimensional basis functions 
$\phi_{p;j}(\vec{v})=
\phi^{u}_{l;j}(u)\phi^{v}_{m;j}(v)\phi^{w}_{n;j}(w)$ 
and the vectors $\vec{v}_{p;j}=(\kappa^{u}_{l;j},\kappa^{v}_{m;j},\kappa^{w}_{n;j})$. Also, $\Delta\vec{v}^{j}= (u_{j}^{R}-u_{j}^{L})(v_{j}^{R}-v_{j}^{L}) (w_{j}^{R}-w_{j}^{L})$, and 
$\omega_{i}:=\omega^{s_{u}}_{l} \omega^{s_{v}}_{m} \omega^{s_{w}}_{n}$, where 
$\omega^{s_{u}}_{l}$,  $\omega^{s_{v}}_{m}$, and $\omega^{s_{w}}_{n}$ are the weights
of the Gauss quadratures of orders $s_u$, $s_v$, and $s_w$, respectively. 
\end{lem}

\subsection{Nodal-DG velocity discretization of the Boltzmann equation}
\label{secN23}

We assume that on each $K_{j}$ the solution to the Boltzmann equation 
is sought in the form 
\begin{equation}
\label{eq2.3.1}
f(t,\vec{x},\vec{v})|_{K_{j}} = \sum_{i=1,s} f_{i;j}(t,\vec{x})\phi_{i;j}(\vec{v})\, .
\end{equation}
The DG velocity discretization that we shall use follows by substitution of the 
representation (\ref{eq2.3.1}) into (\ref{boltzm1}), multiplication 
of the result by a test basis function, and integration over $K_{j}$. Repeating this 
for all $K_{j}$ and using identities (\ref{eq1.1}) we arrive at
\begin{equation}
\label{discveloblzm}
\partial_{t} f_{i;j}(t,\vec{x}) + \vec{v}_{i;j}\cdot \nabla_{x} f_{i;j}(t,\vec{x}) =
\frac{8}{\omega_{i}\Delta\vec{v}^{j}}I_{\phi_{i;j}}\, ,
\end{equation}
where $I_{\phi_{i;j}}$ is the projection of the collision operator 
on the basis function $\phi_{i;j}(\vec{v})$:
\begin{equation}
\label{projcoll}
I_{\phi_{i;j}} = \int_{K_{j}}\phi_{i;j}(\vec{v}) I[f](t,\vec{x},\vec{v})\, d\vec{v}\, .
\end{equation} 

\subsection{Reformulation of the Galerkin projection of the collision operator}
\label{secN32}

Similarly to \cite{AlekseenkoJosyula2012,AlekseenkoJosyula2012a,Majorana2011}, we rewrite the 
DG projection of the collision operator $I_{\phi_{i;j}}$ in the form of a bilinear integral operator 
with a time-independent kernel. Specifically, using the well-known identities (see, e.g.,  \cite{Kogan1969}, Section 2.4), and applying the first principles assumption, we have 
\begin{align}
\label{eq3.2.1}
I_{\phi_{i;j}}&= \int_{\R^3}\int_{\R^3} f(t,\vec{x},\vec{v}) f(t,\vec{x},\vec{v}_{1})
\int_{\Sbb^2}(\phi_{i;j}(\vec{v}')-\phi_{i;j}(\vec{v})) b_{\alpha}(\theta) |g|^\alpha \, d\sigma\,
 d\vec{v}_{1}\, d\vec{v} \nonumber \\
{} & = \int_{\R^3}\int_{\R^3} f(t,\vec{x},\vec{v}) f(t,\vec{x},\vec{v}_{1})
 A(\vec{v},\vec{v}_{1};\phi_{i;j})   d\vec{v}_{1}\, d\vec{v}\, ,
\end{align}
where  
\begin{align}
\label{eq3.2.3}
A(\vec{v},\vec{v}_{1};\phi_{i;j})= |g|^\alpha \int_{\Sbb^2} (\phi_{i;j}(\vec{v}')
- \phi_{i;j}(\vec{v})) b_{\alpha}(\theta) \, d\sigma\, .
\end{align}
We notice that kernel $A(\vec{v},\vec{v}_{1};\phi_{i;j})$ is independent of time and can be pre-computed. 
In \cite{AlekseenkoJosyula2012a}, properties of a kernel closely related to $A(\vec{v},\vec{v}_{1};\phi_{i;j})$ 
are considered. In particular, due to the local support of $\phi_{i;j}(\vec{v})$, it is anticipated that 
kernel  $A(\vec{v},\vec{v}_{1};\phi_{i;j})$ will have only $O(M^{5})$ non-zero components for each $\phi_{i;j}(\vec{v})$, 
where $M$ is the number of discrete velocity cells in each velocity dimension. As a result, evaluation of  
(\ref{eq3.2.1}) will require $O(M^{8})$ operations for each spatial point. This number of evaluations 
is very high. However, as we will show later, it can be reduced to $O(M^6)$ operations using symmetries of
$A(\vec{v},\vec{v}_{1};\phi_{i;j})$, the convolution form of (\ref{eq3.2.1}), and the discrete Fourier 
transform. 

We remark that in many numerical formulations of the Boltzmann equation, the collision 
operator is separated into the gain and loss terms. This separation can also be performed 
in (\ref{eq3.2.1}) in the integral over the collision sphere. After the separation, 
the definition of the operator $A(\vec{v},\vec{v}_{1};\phi_{i;j})$ loses the portion $-|g|^\alpha \int_{\Sbb^2} \phi_{i;j}(\vec{v})d \sigma=:-|g|^\alpha \sigma_{T}$ that, in turn, gives rise to the classical collision frequency 
$\nu(t,\vec{x},\vec{v}):= \int_{\R^3} f(t,\vec{x},\vec{v}_{1}) \sigma_{T}|g|^{\alpha} \, d \vec{v}_{1}$ and the loss term. Theoretical properties of the split formulation 
are very similar to that of (\ref{eq3.2.1}), (\ref{eq3.2.3}). 
Moreover, it may be argued that the 
split formulation is better suited for an application of the Fourier transform 
than (\ref{eq3.2.1}), (\ref{eq3.2.3}) because the kernel of the split formulation 
is decreasing at infinity while kernel $A(\vec{v},\vec{v}_{1};\phi_{i;j})$ 
defined by (\ref{eq3.2.3}) is increasing at infinity in some directions. Also, the 
numerical algorithms introduced in this paper can be extended to the split formulation as well.
However, it was observed that simulations using gain and loss splitting exhibit much 
stronger violation of conservation laws than formulation (\ref{eq3.2.1}), (\ref{eq3.2.3}). 
The exact mechanism of why the non-split formulation preserves the conservation 
laws better is still not clear to the authors. Some insight 
can be obtained by noticing that values of the 
collision kernel  $A(\vec{v},\vec{v}_{1};\phi_{i;j})$ span several orders of magnitude 
and that small values of $A(\vec{v},\vec{v}_{1};\phi_{i;j})$ occur in
sufficiently many points so that they are important collectively. It is possible that 
when small and large values are combined during the evaluation of the gain term, the 
accuracy of the small values is lost or essentially diminished. When the loss term is subtracted from the 
gain term, cancellation occurs producing large errors. On the contrary, 
conservation 
laws are satisfied point-wise in the form (\ref{eq3.2.1}), 
(\ref{eq3.2.3}) up to a small number of algebraic manipulations with the basis functions 
$\phi_{i;j}(\vec{v})$. Because of these considerations, we chose to 
use the non-split form of the collision operator in simulations. 

\subsection{Shift invariance property of kernel $A(\vec{v},\vec{v}_{1};\phi_{i;j})$}
\label{secN33}
\begin{lem}
\label{lem1} 
Let operator $A(\vec{v},\vec{v}_{1};\phi_{i;j})$ be defined by (\ref{eq3.2.3}). Then $\forall\xi\in \R^3$
\begin{equation*}
A(\vec{v}+\vec{\xi},\vec{v}_{1}+\vec{\xi};\phi_{i;j}(\vec{v}-\vec{\xi}))=
A(\vec{v},\vec{v}_{1};\phi_{i;j}) \, .
\end{equation*}
\end{lem}

\proof
Consider $A(\vec{v}+\vec{\xi},\vec{v}_{1}+\vec{\xi};\phi_{i;j}(\vec{v}-\vec{\xi}))$. We clarify that these notations mean that
particle velocities $\vec{v}$ and $\vec{v}_{1}$ in (\ref{eq3.2.3}) are replaced with $\vec{v}+\vec{\xi}$ and $\vec{v}_{1}+\vec{\xi}$
correspondingly and that basis function $\phi_{i;j}(\vec{v})$ is replaced with a ``shifted'' function 
$\phi_{i;j}(\vec{v}-\vec{\xi})$. We notice that the relative speed of the molecules with velocities $\vec{v}+\vec{\xi}$ 
and $\vec{v}_{1}+\vec{\xi}$ is still $\vec{g}=\vec{v}+\vec{\xi}-(\vec{v}_{1}+\vec{\xi}_{1})=\vec{v}-\vec{v}_{1}$. 
The post-collision velocities for the pair of particles will be $\vec{v}'+\vec{\xi}$ and $\vec{v}'_{1}+\vec{\xi}$, where $\vec{v}'$ and $\vec{v}'_{1}$ are given by (\ref{eq2.1.3}). We notice, in particular, that choices of $\theta$ and $\varepsilon$ in 
(\ref{eq2.1.3}) are not affected by $\vec{\xi}$. The rest of the statement follows by a direct substitution:
\begin{align*}
A(\vec{v}&+\vec{\xi},\vec{v}_{1}+\vec{\xi};\phi_{i;j}(\vec{v}-\vec{\xi})) = |g|^{\alpha} \int_{\Sbb^2} \phi_{i;j}((\vec{v}'+\vec{\xi})-\vec{\xi}) b_{\alpha}(\theta)\, d\sigma\,= |g|^{\alpha} \int_{\Sbb^2} \phi_{i;j}(\vec{v}') b_{\alpha}(\theta)\, d\sigma \\
&  = A(\vec{v},\vec{v}_{1};\phi_{i;j})\, . 
\end{align*}
\endproof

We remark that Lemma~{\ref{lem1}} holds for all potentials of molecular 
interaction used in rarefied gas dynamics. This property was used in 
\cite{AlekseenkoJosyula2012a,GambaZhang2014}  to reduce the storage requirement 
for $A(\vec{v},\vec{v}_{1};\phi_{i;j})$
on uniform partitions. 

\subsection{Re-writing the collision operator in the form of a convolution}
\label{secN34}
We will assume that the domain in the velocity space is a rectangular parallelepiped 
and that partition cells are uniform and that the same basis functions are used on 
each cell. In \cite{AlekseenkoNguyenWood2015} it was shown that, in this case, 
the Galerkin projection of the collision operator can be naturally 
re-formulated as a convolution. For convenience, we recall the reasoning here. 

We select a partition cell $K_c$ and designate this cell as the generating cell. Similarly, 
the basis functions $\phi_{i;c}(\vec{v})$ on $K_c$ are designated as the generating basis
functions. Basis functions $\phi_{i;j}(\vec{v})$ on other cells can be obtained using 
a shift in the velocity variable, namely $\phi_{i;j} (\vec{v})=
\phi_{i;c} (\vec{v}+\vec{\xi_{j}})$ where $ \vec{\xi}_{j} \in \mathbb{R}^3$ is the vector 
that connects the center of $K_j$ to the center of $K_c$.

According to  Lemma~{\ref{lem1}}, operator $A(\vec{v},\vec{v}_1,\phi_{i;j})$ is invariant with respect to translations. Therefore
\begin{align}
\label{eq3.1}
I_{\phi_{i;j}} &= \int_{\R^3}\int_{\R^3} f(t,\vec{x},\vec{v})f(t,\vec{x},\vec{v}_{1})A(\vec{v}+\vec{\xi_{j}},\vec{v}_1+\vec{\xi_{j}}; \phi_{i;j}(\vec{u} - \vec{\xi_{j}})) \, d\vec{v}_{1} d\vec{v} \notag\\
 &= \int_{\R^3}\int_{\R^3} f(t,\vec{x},\vec{v})f(t,\vec{x},\vec{v}_{1})A(\vec{v}+\vec{\xi_{j}},\vec{v}_{1}+\vec{\xi_{j}}; \phi_{i;c}(\vec{u})) \, d\vec{v} d\vec{v_{1}} \, .
\end{align}
Performing the substitutions $\vec{\hat{v}} = \vec{v}+\vec{\xi_{j}} $ and $ \vec{\hat{v}}_1 = \vec{v}_{1}+\vec{\xi_{j}}$ in (\ref{eq3.1}), we have
\begin{equation*}
I_{\phi_{i;j}} = \int_{\R^3}\int_{\R^3} f(t,\vec{x},\vec{\hat{v}}-\vec{\xi_{j}}) f(t,\vec{x},\vec{\hat{v}}_1-\vec{\xi_{j}}) A(\vec{\hat{v}},\vec{\hat{v}}_{1};\phi_{i;c}(\vec{u})) \, d\vec{\hat{v}} d\vec{\hat{v}}_1 \, .
\end{equation*}
We then introduce a bilinear convolution operator, $i=1,\ldots,s$
\begin{equation}
\label{eq3.3}
I_i(\vec{\xi}) = \int_{\R^3}\int_{\R^3} f(t,\vec{x},\vec{v}-\vec{\xi})f(t,\vec{x},\vec{v}_{1}-\vec{\xi})A(\vec{v},\vec{v}_{1}; \phi_{i;c}) \, d\vec{v} d\vec{v}_{1} \, ,
\end{equation}
and notice that $I_{\phi_{i;j}}$ can be obtained from (\ref{eq3.3}) as
$I_{\phi_{i;j}}=I_i(\vec{\xi}_j)$. In the following, we will refer to (\ref{eq3.3}) 
as the convolution form of the Galerkin projection of the collision integral.

\section{Discretization of Collision Integral and Fast Evaluation of Discrete Convolution}
\label{secN3}

To evaluate the collision operator numerically, the three dimensional 
integrals in (\ref{eq3.3}) are replaced with the Gauss quadratures 
associated with the nodal-DG 
discretization (\ref{eq01a}). As is discussed above, we are only interested in 
computing convolution (\ref{eq3.3}) at vectors 
$\vec{\xi}=\vec{\xi}_{j}$ that connect centers of the velocity cells $K_j$ to the 
center of the velocity cell $K_c$, the support of $\phi_{i;c}(\vec{v})$. Since the 
same nodal points are used on all velocity cells, shifts $\vec{\xi}_{j}$ 
translate nodal points in one cell to nodal points 
in another cell. As a result, the quadrature sums to evaluate convolution (\ref{eq3.3})
use values of the unknown $f(t,\vec{x},\vec{v})$ at the nodal points
only. In fact, the shift in the velocity variable 
$\vec{v}_{i;l}-\vec{\xi}_{j}$ will correspond 
to a shift in the three dimensional index of the velocity cell 
which we will write formally as $l-j$, 
producing the velocity node $\vec{v}_{i;l-j}(\vec{v})$. The exact expression 
for the shift $l-j$ will be made clear later by considering the cell indices 
separately for 
each velocity dimension. The index $i$ of the node within the cell is not 
affected by the shift. 

We can write the discrete form of (\ref{eq3.3}) as 
\begin{equation}
\label{I+_dist}
I_{i;j}:=I_{i}(\vec{\xi}_{j}) = \sum_{i',i''=1}^s  \sum_{j'=1}^{M^3} \sum_{j''=1}^{M^3} f_{i';j'-j} f_{i'';j''-j} A_{i,i',i'';j',j''} 
\end{equation}
where $f_{i';j'-j}=f(t,\vec{x},\vec{v}_{i';j'-j})$, 
$A_{i,i',i'';j',j''}=A(\vec{v}_{i';j'},\vec{v}_{i'';j''}; \phi_{i;c})(\omega_{i'}\Delta \vec{v}/8)(\omega_{i''}\Delta \vec{v}/8)$ and the three dimensional 
indices $i'$ and $i''$ run over the velocity nodes within a single velocity cell 
and indices $j'$ and $j''$ run over all velocity cells. We note that some shifted 
indices $j'-j$ point outside of the velocity domain. 
In \cite{AlekseenkoJosyula2012a} values outside of the domain 
were substituted with zeros. In cases when the support of the solution 
was well contained within the computational domain, this assumption did not lead to 
large numerical errors. 

\subsection{Discrete convolution form of the collision integral}
We note that equation (\ref{I+_dist}) is a convolution of multi-indexed sequences. 
To make this convolution explicit, we separate the three dimensional indices 
$j=(j_u,j_v,j_w)$, $j'=(j_u',j_v',j_w')$, and $j''=(j_u'',j_v'',j_w'')$ into their directional components 
to obtain
\begin{equation}
\label{I+_dist1}
I_{i;j_{u},j_{v},j_{w}}= \sum_{i',i''=1}^s \sum_{j_u',j_v',j_w'=0}^{M-1} \sum_{j_u'',j_v'',j_w''=0}^{M-1} f_{i';j'_{u}-j_{u},j'_{v}-j_{v},j'_{w}-j_{w}} f_{i'';j''_{u}-j_{u},j''_{v}-j_{v},j''_{w}-j_{w}} A_{i,i',i'';j'_{u},j'_{v},j'_{w},j''_u,j''_{v},j''_{w}}\, ,
\end{equation}
where the components of the index shift $j=(j_u,j_v,j_w)$ are the integer numbers determining 
the shift vector $\vec{\xi}_{j}=(j_u \Delta u, 
j_v \Delta v, j_w \Delta w)$. Here  
$\Delta u$, $\Delta v$, and $\Delta w$ are the dimensions of the uniform velocity cells and $M$ 
is the number of velocity cells in each dimension. For simplicity, we assume that the 
same number of cells is used in each direction. However, the method can be 
formulated for arbitrary numbers of cells.  

Formula (\ref{I+_dist1}) suggests that $O(M^9)$ operations are required 
to compute the collision operator. However, the actual complexity 
of directly evaluating (\ref{I+_dist1}) at all points is $O(M^{8})$
due to the sparsity of $A_{i,i',i'';j',j''}$
\cite{AlekseenkoJosyula2012,AlekseenkoJosyula2012a}. We will show next, 
that an application of discrete Fourier transform allows to evaluate convolution 
(\ref{I+_dist1}) in $O(M^{6})$ operations. 

\subsection{Discrete Fourier transform, circular convolution, and periodic continuation}
Convolution of sequences can be computed efficiently using a fast Fourier transform.
For convenience, let us briefly recall the approach here. 
Let $x_{n}$ and $y_{n}$ be periodic sequences with period $N$. An $N$-point circular 
convolution of $x_{n}$ and $y_{n}$ is defined as (see, e.g., \cite{Nussbaumer1982})
\begin{equation}
\label{cirlconv}
z_{l}=\sum_{n=0}^{N-1} x_{n}y_{l-n}\, .
\end{equation} 
An approach for computing circular convolutions in $O(N\log N)$ 
operations follows from an application of the discrete Fourier transform to 
(\ref{cirlconv}). We recall that the discrete Fourier transform of an $N$-periodic 
sequence $x_{n}$ and its inverse are defined by
\begin{equation}
\label{mdft_1d}
\calF[x]_{k} = \sum_{n=0}^{N-1} W^{ k n} x_{n}, \qquad
x_{l} = \frac{1}{N} \sum_{k=0}^{N-1} W^{-lk} \calF[x]_{k}, \qquad  
\mbox{where}\quad 
W=\e^{-\imath 2\pi/N}\, .
\end{equation}
A well known property of the Fourier transform is that it converts circular 
convolution (\ref{cirlconv}) into a product, namely, 
\begin{equation*}
\calF[z]_{k}=\calF[x]_{k}\calF[y]_{k}\, .
\end{equation*}
Thus to evaluate (\ref{cirlconv}), $\calF[x]_{k}$ and $\calF[y]_{k}$ can be computed in $O(N\log N)$ 
operations using a fast Fourier transform. Then, $\calF[z]_{k}$ are  computed by multiplying 
$\calF[x]_{k}$ and $\calF[y]_{k}$ in $O(N)$ operations. Finally, the values of $z_{l}$ are obtained  
in another $O(N\log N)$ operations by computing the inverse Fourier transform of $\calF[z]_{k}$. 
Convolutions of non-periodic sequences of length $N$ are commonly evaluated using 
a reduction to circular convolutions. For that, sequences are padded with additional 
$N$ zeros to eliminate aliasing and extended to periodic sequences with period $2N$. Then 
a circular convolution of length $2N$ is computed to produce the desired result.

The above approach can also be applied to evaluation of (\ref{I+_dist1}). First, we convert 
(\ref{I+_dist1}) into a multidimensional circular convolution by periodically extending 
$f(t,\vec{x},\vec{v})$ in variable $\vec{v}$ outside of the velocity domain 
and by periodically extending $A(\vec{v},\vec{v}_{1}; \phi_{i;c})$ in both $\vec{v}$ 
and $\vec{v}_{1}$. Without loss of generality, 
we assume that the velocity domain is sufficiently large so that the supports of 
both $f(t,\vec{x},\vec{v})$ and $A(\vec{v},\vec{v}_{1}; \phi_{i;c})$ are limited to at most 
half of the domain's linear size in any direction. This eliminates aliasing when 
treating (\ref{I+_dist1}) as a multidimensional circular convolution. Indeed, this assumption 
does not introduce theoretical difficulties since both solution and the kernel can be  
padded by zeros to a larger region in the velocity space. However, this assumption  
introduces considerable practical difficulties, most notably the larger memory 
requirements for the Fourier image of $A(\vec{v},\vec{v}_{1}; \phi_{i;c})$. Effects of 
truncation and periodic extension of $f(t,\vec{x},\vec{v})$ on the properties of the 
collision operator were considered in \cite{PareschiRusso2000}. Much less is known, however, about the 
effects of truncation and extension of $A(\vec{v},\vec{v}_{1}; \phi_{i;c})$. It can be 
seen from (\ref{eq3.2.3}) that the kernel is growing linearly at the infinity in the 
direction of $\vec{v}-\vec{v}_{1}$ for at least some points $\vec{v}$. Nevertheless, 
a truncation of the kernel $A(\vec{v},\vec{v}_{1}; \phi_{i;c})$ 
was used in \cite{AlekseenkoJosyula2012a} in which entries of $A(\vec{v},\vec{v}_{1}; \phi_{i;c})$ 
were set equal to zero if $\|\vec{v}-\vec{v}_{1}\|<R$, for some selected $R$. 
Numerical simulations in \cite{AlekseenkoJosyula2012a} confirmed that the 
effect of the truncation is negligible if the support of the solution can be 
enclosed in a ball of diameter $R$. Numerical 
experiments conducted in this work also suggest that truncation and periodic extension 
of the solution and the kernel can be performed successfully and the direct 
convolution (\ref{I+_dist1}) can be treated as a circular convolution if the support 
of the solution is sufficiently small.

To obtain the desired formulas for efficient evaluation of the collision operator, we  
apply the discrete Fourier transform (DFT) in indices $(j_u,j_v,j_w)$ to (\ref{I+_dist1}) 
and rewrite the result in a suitable form. As is well documented in similar approaches
(see, e.g., \cite{GambaTharkabhushanam2009, FilbetMouhotPareschi2006, MouhotPareschi2006}), the resulting 
expression is also a convolution that is evaluated directly, but in significantly 
fewer operations. In calculations below we will use the following definition of the 
multidimensional DFT. Let 
$x_{k_1, \dots, k_d}$ be a sequence indexed by $k_1, \dots, k_d$ with equal 
lengths $N$ in each dimension. The DFT $\calF[x]_{k_1,\ldots,k_d}$ of 
$x_{n_1, \dots, n_d}$ is defined as 
\begin{equation}
\label{mdft_def}
\calF[x]_{k_1,\ldots,k_d} = \sum_{n_1=0}^{N-1} \left( W^{ k_1 n_1} 
\sum_{n_2=0}^{N-1} \left( W^{k_2 n_2} \dots 
\sum_{n_d=0}^{N-1} W^{k_d n_d } x_{n_1, \ldots, n_d}
\right) \right)\, .
\end{equation}
Also, it is useful to define the inverse of the transform, 
\begin{equation*}
x_{l_1,\ldots,l_d} = \frac{1}{N} \sum_{k_1=0}^{N-1} \left( W^{ -l_1 k_1} 
\frac{1}{N} \sum_{k_2=0}^{N-1} \left( W^{- l_2 k_{2}} \dots 
\frac{1}{N} \sum_{k_d=0}^{N-1} W^{-l_d k_d } \hat{x}_{k_1,\ldots,k_d} 
\right) \right)\, .
\end{equation*}
Similarly to one dimensional case, fast discrete Fourier transforms can be defined to 
evaluate the transform and its inverse in $O(N^d\log N)$ operations.

\subsection{Formulas for fast evaluation of the collision operator}
To derive the formula for computing the collision operator we rewrite (\ref{I+_dist1}) as
\begin{equation}
\label{I+_distq1}
I_{i;j_{u},j_{v},j_{w}}= \sum_{i',i''=1}^s I_{i,i',i'';j_{u},j_{v},j_{w}} \, ,
\end{equation}
where 
\begin{equation*}
I_{i,i',i'';j_{u},j_{v},j_{w}}=  \sum_{j_u',j_v',j_w'=0}^{M-1} \sum_{j_u'',j_v'',j_w''=0}^{M-1} f_{i';j'_{u}-j_{u},j'_{v}-j_{v},j'_{w}-j_{w}} f_{i'';j''_{u}-j_{u},j''_{v}-j_{v},j''_{w}-j_{w}} A_{i,i',i'';j'_{u},j'_{v},j'_{w},j''_u,j''_{v},j''_{w}}\, .
\end{equation*}
In view of (\ref{I+_distq1}), we can focus on evaluation of $I_{i,i',i'';j_{u},j_{v},j_{w}}$. 
To simplify the notations in the discussion below, we drop the $i$, $i'$, and $i''$ subscripts from 
$I_{i,i',i'';j_{u},j_{v},j_{w}}$, $f_{i';j'_{u},j'_{v},j'_{w}}$, 
$A_{i,i',i'';j'_{u},j'_{v},j'_{w},j''_u,j''_{v},j''_{w}}$, and  $\calF[I_{i,i',i''}]_{k_{u},k_{v},k_{w}}$ 
and write $I_{j_u,j_v,j_w}$, $f_{j'_{u},j'_{v},j'_{w}}$, 
$A_{j'_{u},j'_{v},j'_{w},j''_u,j''_{v},j''_{w}}$, and $\calF[I]_{k_{u},k_{v},k_{w}}$, respectively. 
In particular, we have
\begin{equation}
\label{I_all_index}
I_{j_u,j_v,j_w}=  \sum_{j_u',j_v',j_w'=0}^{M-1} \sum_{j_u'',j_v'',j_w''=0}^{M-1}f_{j'_{u}-j_{u},j'_{v}-j_{v},j'_{w}-j_{w}} f_{j''_{u}-j_{u},j''_{v}-j_{v},j''_{w}-j_{w}} A_{j'_{u},j'_{v},j'_{w},j''_u,j''_{v},j''_{w}}\, .
\end{equation}
As is seen from definition (\ref{mdft_def}), the multi-dimensional DFT results from applying 
the one-dimensional DFT along each dimension of the sequence for fixed values of indices 
in the other dimensions (see e.g., \cite{Nussbaumer1982}).

We fix indices $j_v$ and $j_w$  in equation (\ref{I_all_index}) and apply the one-dimensional DFT 
in the remaining index $j_u$. Using linearity of the DFT and reordering the sums, we have 
\begin{equation*}
\calF[I_{j_v,j_w}]_{k_u} = \sum_{j_v',j'_w=0}^{M-1} \sum_{j_v'',j''_{w}=0}^{M-1}  \calF[\tilde{I}_{j_v,j_v',j_v'',j_w,j_w',j_w''}]_{k_u}\, ,
\end{equation*}
where 
\begin{equation*}
\calF[\tilde{I}_{j_v,j_v',j_v'',j_w,j_w',j_w''}]_{k_u} = 
\sum_{j_u=0}^{M-1}  W^{ k_u j_u} \tilde{I}_{j_u;j_v,j_v',j_v'',j_w,j_w',j_w''} \, ,
\end{equation*}
\begin{equation*}
\tilde{I}_{j_u;j_v,j_v',j_v'',j_w,j_w',j_w''} = \sum_{j_u''=0}^{M-1} \sum_{j_u'=0}^{M-1} f_{j'_{u}-j_{u},j'_{v}-j_{v},j'_{w}-j_{w}} f_{j''_{u}-j_{u},j''_{v}-j_{v},j''_{w}-j_{w}} A_{j'_{u},j'_{v},j'_{w},j''_u,j''_{v},j''_{w}}\, .
\end{equation*}
Once again, for purposes of calculating the one-dimensional DFT along dimension $j_u$, we need only consider the transform of $\tilde{I}_{j_u,j_v,j_v',j_v'',j_w,j_w',j_w''}$. 
By similar argument as before, we fix and drop indices $j_v$, $j_v'$, $j_v''$, $j_w$, $j_w'$, and $j_w''$ in the latter formula and write 
\begin{align}
\tilde{I}_{j_u} &= \sum_{j_u''=0}^{M-1} \sum_{j_u'=0}^{M-1} f_{j'_{u}-j_{u}} f_{j''_{u}-j_{u},} A_{j'_{u},j''_u}\, , \label{I_onedim_u} \\
\calF[\tilde{I}]_{k_u} &= \sum_{j_u=0}^{M-1} \sum_{j_u''=0}^{M-1} \sum_{j_u'=0}^{M-1} W^{k_u j_u} f_{j'_{u}-j_{u}} f_{j''_{u}-j_{u},} A_{j'_{u},j''_u}\, . \label{FI_onedim_u}
\end{align}
We note that evaluating $\calF[\tilde{I}]_{k_u}$ directly would require $O(M^3)$ operations.
However, taking into consideration the discussion in the last section, expression in the right side of (\ref{I_onedim_u}) can be considered as a circular convolution, having a form similar to (\ref{cirlconv}). This motivates us to explore properties of the DFT and rewrite (\ref{FI_onedim_u}) in a form suitable for numerical computation. This is accomplished in the following lemma.

\begin{lem}
\label{lemma_dft}
Let $\{ f_j\}_{j=0}^{M-1}$ be a $M$ periodic sequence and $\{A_{ij}\}_{i,j=0}^{M-1}$ be a two index sequence that is $M$ periodic in both its indices. Let $\{\tilde{I}_j\}_{j=0}^{M-1}$ be a new sequence defined by 
\begin{equation}
\label{I_onedim}
\tilde{I}_{j} =  \sum_{j'=0}^{M-1} \sum_{j''=0}^{M-1} f_{j'-j}f_{j''-j} A_{j',j''}\, .
\end{equation}
Let $\calF[\tilde{I}]_k$ be the DFT of $\tilde{I}_j$, then  
\begin{equation}
\label{onedfconv}
\calF[\tilde{I}]_k = M \sum_{l=0}^{M-1} \calF^{-1}[f]_{k-l} \calF^{-1}[f]_{l} \calF[A]_{k-l,l}\, .
\end{equation}
\end{lem}

\begin{proof}
Applying the one-dimensional DFT to $\tilde{I}_j$, we have
\begin{equation}
\label{FI_u_index}
\calF[\tilde{I}]_k = \sum^{M-1}_{j_u=0} \sum^{M-1}_{j'=0} \sum^{M-1}_{j''=0} W^{k j } f_{j'-j}f_{j''-j}A_{j',j''}\, .
\end{equation}
We define $\calF[A_{j'}]_{l}$ to be the one-dimensional DFT of $A_{j',j''}$ in the second index, i.e.,
\begin{align*}
\calF[A_{j'}]_{l} = \sum_{j''=0}^{M-1} W^{ j'' l} A_{j',j''}\, ,
\end{align*}
and rewrite $A_{j',j''}$ as
\begin{align}
\label{one_dim_FA}
A_{j',j''}= \frac{1}{M} \sum_{l=0}^{M-1} W^{ -j'' l} \calF[A_{j'}]_{l}\, .
\end{align}
Substituting (\ref{one_dim_FA}) into (\ref{FI_u_index}), we have
\begin{align}
\label{original}
\calF [\tilde{I}]_k &= \frac{1}{M} \sum_{j=0}^{M-1} \sum_{j'=0}^{M-1} \sum_{j''=0}^{M-1} W^{jk} f_{j'-j}f_{j''-j}
\left( \sum_{l=0}^{N-1} W^{-j'' l} \calF[A_{j'}]_{l} \right)  \nonumber \\
&= \frac{1}{M} \sum_{l=0}^{M-1} \sum_{j=0}^{M-1} \sum_{j'=0}^{M-1} \sum_{j''=0}^{M-1} W^{jk} f_{j'-j}f_{j''-j} W^{-j'' l } \calF[A_{j'}]_{l} 
\end{align}
Consider the sum that runs over index $j''$. Assuming that indices $l$, $j$, and $j'$ are held constant, we split the sum into two parts.
\begin{align}
\label{j''_split}
\sum_{j''=0}^{M-1} &W^{jk} f_{j'-j}f_{j''-j} W^{-j'' l } \calF[A_{j'}]_{l}  \nonumber \\
=&\sum_{j''=0}^{j-1} W^{jk} f_{j'-j}f_{j''-j} W^{-j'' l } \calF[A_{j'}]_{l} + \sum_{j''=j}^{M-1} W^{jk} f_{j'-j}f_{j''-j} W^{-j'' l } \calF[A_{j'}]_{l} \, .
\end{align}

Notice $j''-j < 0$ in the first sum. Using periodicity of $f_{j}$, the summation index can be 
redefined so that only values $f_{j''-j}$ with positive $j''-j$ appear in the sum. Indeed, 
we assume that $j'' < j$ and notice that $f_{j''-j}=f_{j''-j+M}$ since $f_{j}$ is $M$ periodic. 
We also have $W^{M}=1$, so $W^{j'' l } = W^{j''l+Ml} = W^{(j''+M)l}$. 
Introducing $\hat{j}''=j''+M$, we observe 
\begin{align*}
\sum_{j''=0}^{j-1} W^{jk} f_{j'-j}f_{j''-j} W^{-j'' l } \calF[A_{j'}]_{l} &
=\sum_{j''=0}^{j-1} W^{jk} f_{j'-j}f_{j''-j+M} W^{(-j''+M) l } \calF[A_{j'}]_{l} \\
&=\sum_{\hat{j}''=M}^{M-1+j} W^{jk} f_{j'-j}f_{\hat{j}''-j} W^{-\hat{j}'' l } \calF[A_{j'}]_{l}\, . 
\end{align*}
Combining the last formula with (\ref{j''_split}) we have 
\begin{align}
\label{j''_split2}
\sum_{j''=0}^{M-1} W^{jk} f_{j'-j}f_{j''-j} W^{-j'' l } \calF[A_{j'}]_{l} 
&= \sum_{j''=j}^{M-1+j} W^{jk} f_{j'-j}f_{j''-j} W^{-j'' l } \calF[A_{j'}]_{l} \, .
\end{align}
Introducing a substitution of index $u''=j''-j$ we rewrite the right side of  (\ref{j''_split2}) as follows 
\begin{align*}
\sum_{j''=0}^{M-1} W^{jk} f_{j'-j}f_{j''-j} W^{-j'' l } \calF[A_{j'}]_{l} 
= \sum_{u''=0}^{M-1} W^{jk} f_{j'-j}f_{u''} W^{(-u''-j) l } \calF[A_{j'}]_{l}\, .
\end{align*}
Going back to (\ref{original}), we replace the inside sum with the last expression to have
\begin{align}
\label{original_u''}
\frac{1}{M} \sum_{l=0}^{M-1} &\sum_{j=0}^{M-1} \sum_{j'=0}^{M-1} \sum_{u''=0}^{M-1} W^{jk} f_{j'-j}f_{u''} W^{(-u''-j) l } \calF[A_{j'}]_{l} \nonumber \\
&= \frac{1}{M} \sum_{l=0}^{M-1} \sum_{j'=0}^{M-1} \sum_{u''=0}^{M-1} W^{-u'' l }f_{u''} \left( \sum_{j=0}^{M-1} W^{j(k-l)} f_{j'-j}  \calF[A_{j'}]_{l}  \right)\, .
\end{align}

Now we focus on the term within the parentheses in (\ref{original_u''}). Splitting the sum and using periodicity, we obtain
\begin{align}
\sum_{j=0}^{M-1}& W^{j(k-l)} f_{j'-j} \calF[A_{j'}]_{l} \\
&=\sum_{j=0}^{j'} W^{j(k-l)} f_{j'-j} \calF[A_{j'}]_{l} + \sum_{j=j'+1}^{M-1} W^{j(k-l)} f_{j'-j} \calF[A_{j'}]_{l}  \nonumber \\
&=\sum_{j=0}^{j'} W^{j(k-l)} f_{j'-j} \calF[A_{j'}]_{l} + \sum_{j=j'+1}^{M-1} W^{(j-M)(k-l)} f_{j'-j+M} \calF[A_{j'}]_{l}  \nonumber \\
&=\sum_{j=0}^{j'} W^{j(k-l)} f_{j'-j} \calF[A_{j'}]_{l} + \sum_{\hat{j}=j'-M+1}^{-1} W^{\hat{j}(k-l)} f_{j'-\hat{j}} \calF[A_{j'}]_{l}  \nonumber  \\
&=\sum_{j=j'-M+1}^{j'} W^{j(k-l)} f_{j'-j} \calF[A_{j'}]_{l} 
= \sum_{u'=0}^{M-1} W^{(j'-u')(k-l)} f_{u'} \calF[A_{j'}]_{l}\, .
\end{align}
Here $u'=j'-j$. 
Substituting this result into (\ref{original_u''}) and regrouping sums, we yield
\begin{align}
\label{grouped}
\frac{1}{M} \sum_{l=0}^{M-1} &\sum_{j'=0}^{M-1} \sum_{u''=0}^{M-1} W^{-u'' l }f_{u''} \left( \sum_{j=0}^{M-1} W^{j(k-l)} f_{j'-j}  \calF[A_{j'}]_{l}  \right) \nonumber \\
&= \frac{1}{M} \sum_{l=0}^{M-1} \sum_{j'=0}^{M-1} \sum_{u''=0}^{M-1} W^{-u'' l}f_{u''} \left(\sum_{u'=0}^{M-1} W^{(j'-u')(k-l)} f_{u'} \calF[A_{j'}]_{l} \right) \nonumber \\
&= M\sum_{l=0}^{M-1} \left( \frac{1}{M} \sum_{u'=0}^{M-1} W^{-u'(k-l)} f_{u'}  \right) \left( \frac{1}{M} \sum_{u''=0}^{M-1} W^{-u'' l} f_{u''} \right) \left(  \sum_{j'=0}^{M-1} W^{j'(k-l)} \calF[A_{j'}]_{l} \right).
\end{align}
The terms in the parentheses in (\ref{grouped}) are just the definitions of the DFT. Thus we can write the equation as
\begin{align*}
\calF [\tilde{I}]_k = M \sum_{l=0}^{M-1} \calF^{-1}[f]_{k-l} \calF^{-1}[f]_l \calF[A]_{k-l,l} \, .
\end{align*}
\end{proof}

Lemma~\ref{lemma_dft} allows us to compute (\ref{FI_onedim_u}) in $O(M^2)$ operations. Indeed, it takes 
$O(M \log M)$ operations to compute $\calF^{-1}[f]_{k_{u}}$ using a fast Fourier transform and it takes 
$O(M^2)$ operations to compute discrete convolution in the frequency space (\ref{onedfconv}). To extend this result to $\calF[I]_{k_{u},k_{v},k_{w}}$, it is sufficient to repeat the approach for indices $j_{v}$ and $j_{w}$ focusing on one dimension at a time. The following theorem summarizes the result. 

\begin{thm}
\label{thm3.1}
Let $f_{j_{u},j_{v},j_{w}}$ be a three-index sequence that is periodic in each index with period $M$ and let $A_{j'_{u},j'_{v},j'_{w},j''_{u},j''_{v},j''_{w}}$ be a $M$-periodic six-dimensional tensor. The multi-dimensional discrete Fourier transform of equation (\ref{I_all_index}) can be represented as
\begin{equation}
\label{FI_formula}
\calF [I]_{k_{u},k_{v},k_{w}} = M^3 \sum_{l_{u},l_{v},l_{w}=0}^{M-1} \calF^{-1}[f]_{k_{u}-l_{u},k_{v}-l_{v},k_{w}-l_{w}} \calF^{-1}[f]_{l_{u},l_{v},l_{w}} \calF[A]_{k_{u}-l_{u},k_{v}-l_{v},k_{w}-l_{w},l_{u},l_{w},l_{w}}
\end{equation}
\end{thm}

\begin{proof}
We apply the one dimensional discrete Fourier transform along $j_u$ in equation (\ref{I_all_index}) and apply Lemma {\ref{lemma_dft}}:
\begin{align}
\label{lemma_once}
\calF[I_{j_v,j_w}]_{k_{u}}&= \sum_{j_v',j_w'=0}^{M-1} \sum_{j_v'',j_w''=0}^{M-1} \left(\sum_{j_u=0}^{M-1} \sum_{j_u'=0}^{M-1} \sum_{j_u''=0}^{M-1} W^{j_u k} f_{j_u'-j_u,j_v'-j_v,j_w'-j_w} f_{j_u''-j_u,j_v''-j_v,j_w''-j_w} A_{j'_{u},j'_{v},j'_{w},j''_u,j''_{v},j''_{w}}\right) \nonumber \\ 
&= M\sum_{j_v',j_w'=0}^{M-1} \sum_{j_v'',j_w''=0}^{M-1} \sum_{l_{u}=0}^{M-1} \calF^{-1}[f_{j_v'-j_v,j_w'-j_w}]_{k_{u}-l_{u}} \calF^{-1}[f_{j_v''-j_v,j_w''-j_w}]_{l_u} \calF[A_{j'_{v},j'_{w},j''_{v},j''_{w}}]_{k_{u}-l_{u},l_{u}} \nonumber \\ 
&= M\sum_{l_u=0}^{M-1} \sum_{j_w'=0}^{M-1} \sum_{j_w''=0}^{M-1} \nonumber \\
& \hspace{10mm} 
\left( \sum_{j_v'=0}^{M-1} \sum_{j_v''=0}^{M-1}   \calF^{-1}[f_{j_v'-j_v,j_w'-j_w}]_{k_{u}-l_{u}} \calF^{-1}[f_{j_v''-j_v,j_w''-j_w}]_{l_{u}} \calF[A_{j'_{v},j'_{w},j''_{v},j''_{w}}]_{k_{u}-l_{u},l_{u}} \right) .
\end{align}

We now focus on the terms inside the parentheses. We fix the indices $j_w$, $j_w'$, $j_w''$, $k_{u}$, and $l_{u}$ in the grouped terms. We drop these indices and write
\begin{align*}
\tilde{I}_{j_v} =  \sum_{j_v'=0}^{M-1} \sum_{j_v''=0}^{M-1} \tilde{f}_{j_v'-j_v} \tilde{f}_{j_v''-j_v} \tilde{A}_{j_v',j_v''}\, , 
\end{align*}
where 
\begin{equation*}
\tilde{f}_{j_v'-j_v} = \calF^{-1}[f_{j_v'-j_v}],\qquad 
\tilde{A}_{j_v',j_v''} =  \calF[A_{j'_{v},j''_{v}}].
\end{equation*}
We can see that this expression is identical to (\ref{I_onedim}). We take the discrete Fourier transform along the $j_v$ index of $\calF[\tilde{I}]_{k_v}$ and apply Lemma~\ref{lemma_dft} to arrive at
\begin{align}
\label{dft_dft}
\calF[\tilde{I}]_{k_v} &= M \sum_{l_{v}=0}^{M-1} \calF^{-1}[ \tilde{f}]_{k_{v}-l_{v}} \calF^{-1}[\tilde{f}]_{l_{v}} \calF[\tilde{A}]_{k_{v}-l_{v},l_{v}}\, .
\end{align}
We recall definitions of $\tilde{f}_{j_v'-j_v}$ and $\tilde{A}_{j_v',j_v''}$ and notice that the 
multi-index  Fourier transform results from applying the one-dimensional transform in each index. 
Bringing indices
$j_w'$, $j_w''$, $k_{u}$, and $l_{u}$ back, equation (\ref{dft_dft}) becomes
\begin{equation*}
\calF[I_{j_w}]_{k_u,k_v} = M^2 \sum_{l_u,l_v=0}^{M-1} \calF^{-1}[f_{j'_w-j_w}]_{k_u-l_u,k_v-l_v} \calF^{-1}[f_{j''_w-j_w}]_{l_u,l_v} \calF[A_{j'_w,j''_w}]_{k_u-l_u,k_w-l_w,l_u,l_w}\, .
\end{equation*}
Performing the discrete Fourier transform in the $j_w$ and repeating the argument once more we arrive at the statement of the theorem.
\end{proof}

\section{The Algorithm and its Complexity}
\label{secN4}

Theorem~\ref{thm3.1} allows us to calculate the collision operator 
(\ref{I+_dist1}) in $O(s^9 M^6)$ operations using the algorithm outlined below. 
We note that $\calF[A_{i,i',i''}]_{k_{u},k_{v},k_{w},l_{u},l_{w},l_{w}}$ 
can be precomputed and therefore does not factor into the algorithmic complexity analysis. 
\begin{enumerate}
	\item The first step of the algorithm is to evaluate $\calF^{-1}[f_i]_{k_u,k_v,k_w}$. Evaluation of the inverse Fourier transform requires $O(M^3 \log M)$ operations for each value of index $i$ by utilizing three-dimensional FFT. This must be repeated for each $i$, resulting in the total of $O(s^3 M^3 \log M)$ operations where $s^3$ is the number of velocity nodes in each velocity cell.
    \item Next we directly compute the convolution
	\begin{equation*}
		\calF [I_{i,i',i''}]_{k_{u},k_{v},k_{w}} = M^3 \sum_{l_{u},l_{v},l_{w}=0}^{M-1} \calF^{-1}[f_{i'}]_{k_{u}-l_{u},k_{v}-l_{v},k_{w}-l_{w}} \calF^{-1}[f_{i''}]_{l_{u},l_{v},l_{w}} \calF[A_{i,i',i''}]_{k_{u}-l_{u},k_{v}-l_{v},k_{w}-l_{w},l_{u},l_{w},l_{w}}\, 
	\end{equation*}
	using periodicity of both $\calF^{-1}[f_{i}]_{l_{u},l_{v},l_{w}}$ 
	and $\calF[A_{i,i',i''}]_{k_{u},k_{v},k_{w},l_{u},l_{w},l_{w}}$. 
	
    For fixed values of indices $i$, $i'$, $i''$ and $k_u$, $k_v$, $k_w$, calculating 
    $\calF [I_{i,i',i''}]_{k_{u},k_{v},k_{w}}$ requires $O(M^3)$ arithmetic operations. 
    There are $M^3$ combinations of $k_u,k_v,k_w$ and $s^9$ combinations of indices  $i$, $i'$, and $i''$, therefore complexity of this step is $O(s^9 M^6)$.

    \item Linearity of the Fourier transform allows us to sum $\calF[I_{i,i',i''}]_{k_u,k_v,k_w}$ along $i'$, $i''$ to calculate $\calF[I_i]_{k_u,k_v,k_w}$. 
    \begin{align*}
    \calF[I_{i}]_{k_{u},k_{v},k_{w}}
= \sum_{i',i''=1}^s \calF[I_{i,i',i''}]_{k_{u},k_{v},k_{w}}\, . 
    \end{align*}
    This step requires adding $s^6$ sequences of length $M^3$ for every value of $i$, resulting in a complexity of $O(s^9 M^3)$ operations. 
    \item We recover $\calF^{-1}[\calF[I_{i}]]_{j_u,j_v,j_w}=I_{i;j_u,j_v,j_w}$.
    This requires calculating the three-dimensional inverse DFT for every $i$ which gives a complexity of $O(s^3 M^3 \log M)$
\end{enumerate}

Overall, the algorithm has the numerical complexity of $O(s^9 M^6)$ dominated by step 2. 
We note that $s=s_{u}=s_{v}=s_{w}$ can be kept fixed and the number of cells 
$M^3$ in velocity domain can be increased if more accuracy is desired. In this case, the main contribution to 
complexity growth comes from $M$, the number of velocity cells in one velocity dimension. Thus 
we can consider the algorithm to be of 
complexity $O(M^6)$. In our simulations $s_u,s_v,s_w \leq 3$, however moderately higher values 
may be used too. Results of this analysis are validated within the next section.

\section{Numerical Results}
\label{secN5}

In this section we will describe results of numerical experiments for computing 
the collision operator using nodal-DG velocity discretizations and evaluation of 
convolution using the discrete Fourier 
transform. 

Our first discussion is concerned with estimating numerical complexity 
of the method. The analysis of the previous section suggests that the number of 
arithmetic operations to evaluate the collision operator using the Fourier transform is 
$O(M^6)$, where $M$ is the number of the velocity cells in one velocity dimension. The direct 
evaluation of convolution employed in  \cite{AlekseenkoJosyula2012a} requires $O(M^8)$ operations. 
In Table~{\ref{tab01}}, CPU times are listed for evaluating the collision operator at 
one spatial point both using the Fourier transform and directly. The computations were 
performed on an Intel Core i7-3770 3.4 GHz processor. The numbers of cells 
in velocity domain were varied from 9 to 27 in each velocity dimension. The computational 
complexity is modeled 
using the formula $t=O(M^\alpha)$, where $\alpha$ is constant. The observed values of the 
exponent $\alpha$ are computed using the formula $\alpha=\ln(M_{1}/M_{2})/\ln(t_{1}/t_{2})$. 

We note that in the case of the Fourier evaluation, the observed orders are significantly 
higher than the projected value of 6. Still, the orders are significantly lower than the 
orders of the direct evaluation. Deviations from the theoretical estimate of 
$\alpha=6$ may be due to the costs of the memory transfer operations and due to the choice of 
the specific fast Fourier transform that is automatically selected by the KML library based 
on the value of $M$. Overall, the new approach showed a dramatic improvement in speed as 
compared to the direct evaluation of the collision operator used in \cite{AlekseenkoJosyula2012}.
The acceleration is expected to be even larger for higher values of $M$.

\begin{table}[h]
  \begin{tabular}[c]{ c c c c c c}
    \hline 
    & \multicolumn{2}{c|}{DFT} & \multicolumn{2}{c|}{Direct} & Speedup\\
    \hline
    $M$ & time, s & $\alpha$ & time, s & $\alpha$ &  \\
    \hline
    9   & 1.47E-02 &         & 1.25E-01 &       & 8.5  \\
    15  & 3.94E-01 &  6.43   & 4.91E+00 &  7.18 & 12.5 \\
    21  & 3.09E+00 &  6.14   & 7.80E+01 &  8.21 & 25.2 \\
    27  & 1.64E+01 &  6.65   & 6.05E+02 &  8.15 & 36.7 \\
    \hline
  \end{tabular}
\caption{\label{tab01} CPU times for evaluating the collision operator directly and using the Fourier transform.}
\end{table}

\subsection{The gain-loss vs.\ the non-split forms of the collision operator} 
The form of the collision integral (\ref{eq3.2.1}),  (\ref{eq3.2.3}) admits a few 
re-formulations that are worth considering for the purpose of numerical implementation. 
It was observed that the ability of numerical solutions to conserve mass, momentum, 
and energy is strongly affected by the form of the discrete collision integral. 
The first such reformulation consists of representing the numerical solution as
\begin{equation}
\label{dec00}
f(t,\vec{x},\vec{v})=f_{M}(t,\vec{x},\vec{v})+\Delta f(t,\vec{x},\vec{v})\, ,
\end{equation}
where $f_{M}(t,\vec{x},\vec{v})$ is the Maxwellian distribution that at every point $(t,\vec{x})$ has the same density, bulk 
velocity, and temperature as $f(t,\vec{x},\vec{v})$. Also known as the macro-micro decomposition 
(see, e.g., \cite{Filbet2015309}), this representation of the solution was applied in 
\cite{AlekseenkoJosyula2012} to improve conservation properties of the scheme 
when the solution is near continuum. The decomposition (\ref{dec00}) is also used in the 
numerical simulations presented in this paper. For convenience, let us briefly summarise the 
idea of the decomposition. We substitute 
(\ref{dec00}) into (\ref{eq3.2.1}) to obtain an alternative representation of the collision integral:
\begin{align}
\label{dec01}
I_{\phi_{i;j}}&= \int_{\R^3}\int_{\R^3} f(t,\vec{x},\vec{v}) f(t,\vec{x},\vec{v}_{1})
 A(\vec{v},\vec{v}_{1};\phi_{i;j})   d\vec{v}_{1}\, d\vec{v}\, \nonumber \\
 &=\int_{\R^3}\int_{\R^3} [f_{M}(t,\vec{x},\vec{v}) \Delta f(t,\vec{x},\vec{v}_{1})
 +\Delta f(t,\vec{x},\vec{v}) f_{M}(t,\vec{x},\vec{v}_{1})\nonumber \\
 &\hspace*{10mm}{}+ \Delta f(t,\vec{x},\vec{v}) \Delta f(t,\vec{x},\vec{v}_{1})]
 A(\vec{v},\vec{v}_{1};\phi_{i;j})   d\vec{v}_{1}\, d\vec{v}\, . 
\end{align}
Here we used the fact that the collision integral is zero for any 
Maxwellian. All formulas discussed in previous sections and also 
in the following can be easily adjusted to the decomposed form of the solution.

The second possible re-formulation consists of splitting the operator 
$A(\vec{v},\vec{v}_{1};\phi_{i;j})$ given by (\ref{eq3.2.3}) into the loss and gain terms, i.e.,
\begin{align}
\label{dec02}
A(\vec{v},\vec{v}_{1};\phi_{i;j})= 
 \int_{\Sbb^2} \phi_{i;j}(\vec{v}') b_{\alpha}(\theta) |g|^\alpha \, d\sigma - \phi_{i;j}(\vec{v}) \sigma_{T} |g|^{\alpha}\, ,
\end{align}
where  $\sigma_{T}=\int_{\Sbb^2} b_{\alpha} (\theta) d \sigma$. By separating the 
integrals in the velocity variable in (\ref{eq3.2.1}) and performing a substitution 
in the second term, we obtain the split formulation of the collision integral:
\begin{align}
\label{dec03}
I_{\phi_{i;j}}& = \int_{\R^3}\int_{\R^3} f(t,\vec{x},\vec{v}) f(t,\vec{x},\vec{v}_{1})
A^{+}(\vec{v},\vec{v}_{1};\phi_{i;j}) - \int_{\R^{3}} f(t,\vec{x},\vec{v})\phi_{i;j}(\vec{v})\nu(t,\vec{x},\vec{v}) \, d\vec{v},
\end{align} 
where the collision frequency $\nu(t,\vec{x},\vec{v})$ \cite{Struchtrup2005} and the kernel $A^{+}(\vec{v},\vec{v}_{1};\phi_{i;j})$ are 
defined by
\begin{align*}
\nu(t,\vec{x},\vec{v})&= \int_{\R^3} f(t,\vec{x},\vec{v}_{1}) \sigma_{T}|g|^{\alpha} \, d \vec{v}_{1}\quad\mbox{and}\quad
A^{+}(\vec{v},\vec{v}_{1};\phi_{i;j})= |g|^\alpha \int_{\Sbb^2} \phi_{i;j}(\vec{v}') b_{\alpha}(\theta) \, d\sigma\, .
\end{align*}
We note that formulation (\ref{dec03}) has properties that are beneficial in a theoretical study. In particular, $A^{+}(\vec{v},\vec{v}_{1};\phi_{i;j})$ is decreasing at infinity, while $A(\vec{v},\vec{v}_{1};\phi_{i;j})$ is increasing linearly in the direction of 
$\vec{v}-\vec{v}_{1}$ for at least some points $\vec{v}$. Splitting of the 
collision operator into the gain and loss terms was used by many authors for both 
theoretical and numerical studies. 

It was observed, however, that the split form of the collision integral had 
some numerical properties that make it inferior to the non-split form. In 
Figure~\ref{fig01} results of the evaluation of the collision integral 
at a single spatial point are presented for both split and non-split 
formulations. The value of the solution $f(t,\vec{v})$ in these computations 
is given by the sum of two Maxwellian distributions with dimensionless 
densities, bulk velocities, and temperatures given as follows: $n_{1}=1.6094$, 
$n_{2}=2.8628$, $\bar{\vec{u}}_{1}=(0.7750,0,0)$, $\bar{\vec{u}}_{2}=(0.4357,0,0)$,
$T_{1}=0.3$, and $T_{2}=0.464$. These values correspond to upstream and downstream conditions
of a normal shock wave with the Mach number 1.55. Discretization of the solution was done using 
27 velocity cells in each dimension and one velocity node per cell. The collision operator was 
evaluated using both the split and non-split forms and using both direct evaluation and evaluation 
using the Fourier transform. Results of the direct evaluation of the 
split form of the collision operator are shown in plots (b) and (e). Notably, values 
of the collision operator are zero at the 
boundary of the domain, which is what one would expect from the collision process. Results 
of evaluation of the split form of the collision operator using the Fourier transform are shown in plots (a) and (b). 
Significant non-zero values can be observed at the corners of the domain. This is likely to be 
a manifestation of aliasing. Results of evaluating the non-split form of the collision 
operator using the Fourier transform are shown in plots (c) and (f). One can notice that 
in the case of the non-split form, aliasing is not visible. In fact, the 
$L^{1}$-norm of the difference between the direct and Fourier
evaluations of the non-split collision operator in this case was 2.9E-4 and 
the $L^{\infty}$-norm was 
1.1E-4. We note that the diameters of the supports of the collision kernels 
are comparable in both split and non-split cases.  We also note that aliasing errors 
can be reduced by padding the solution and the collision kernel with zeros. However, 
this will also increase memory and time costs of calculations. The non-split form 
of the collision operator has significantly smaller aliasing errors and does not 
require zero padding. Therefore, it is more efficient. 

\begin{center}
\begin{figure}[h]
\centering
  \begin{tabular}{@{}ccc@{}}
  \includegraphics[height=.18\textheight]{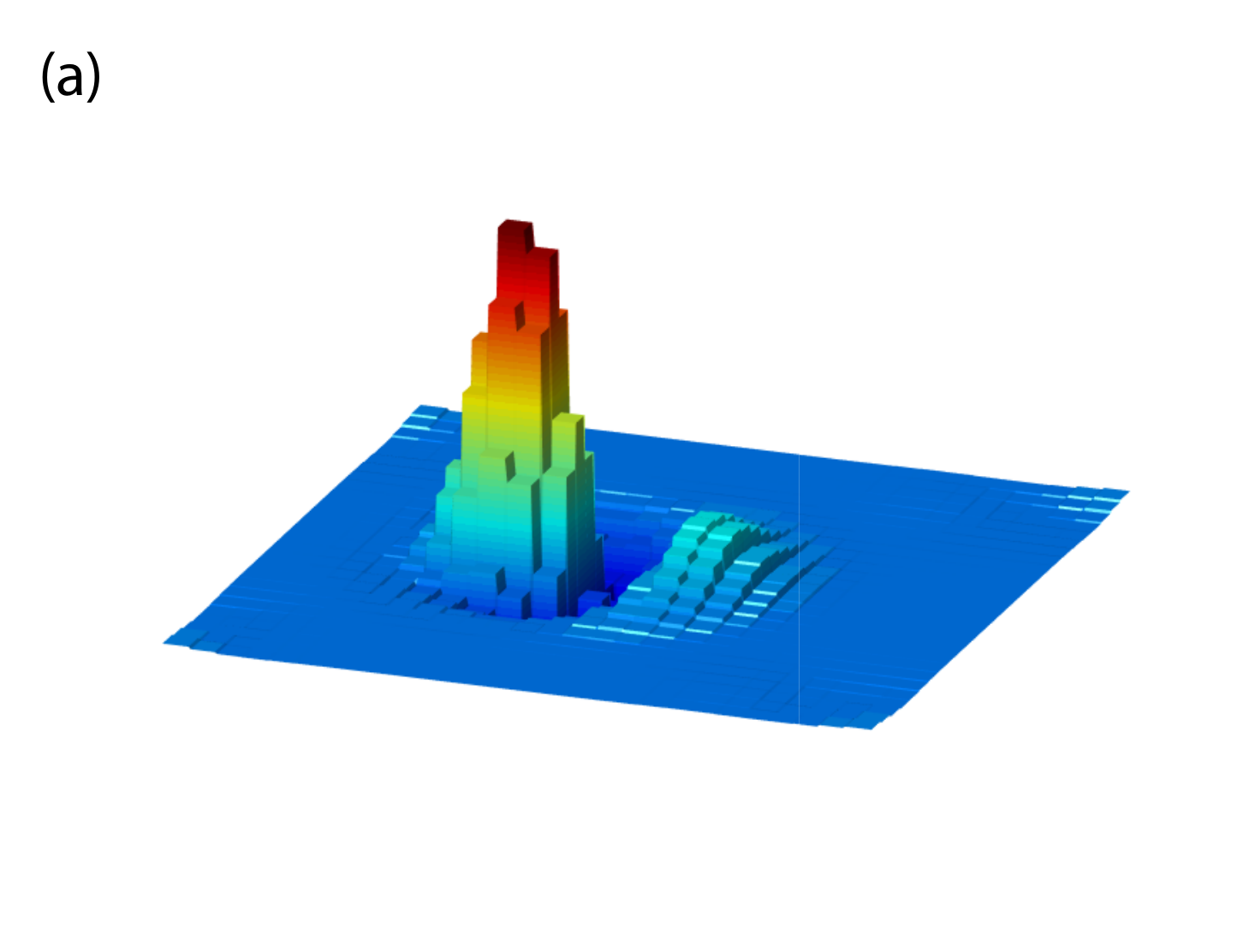}&
  \includegraphics[height=.18\textheight]{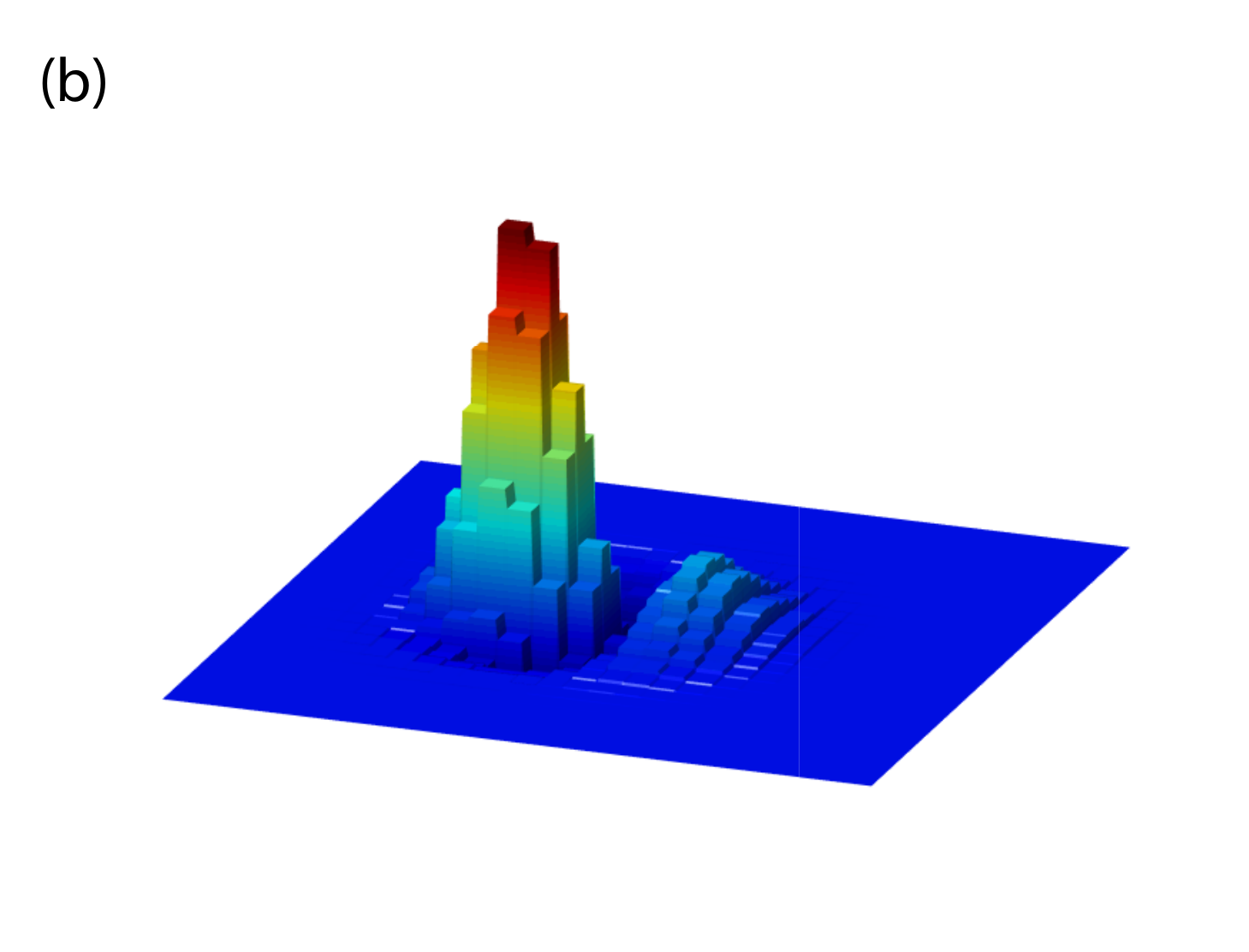}&
  \includegraphics[height=.18\textheight]{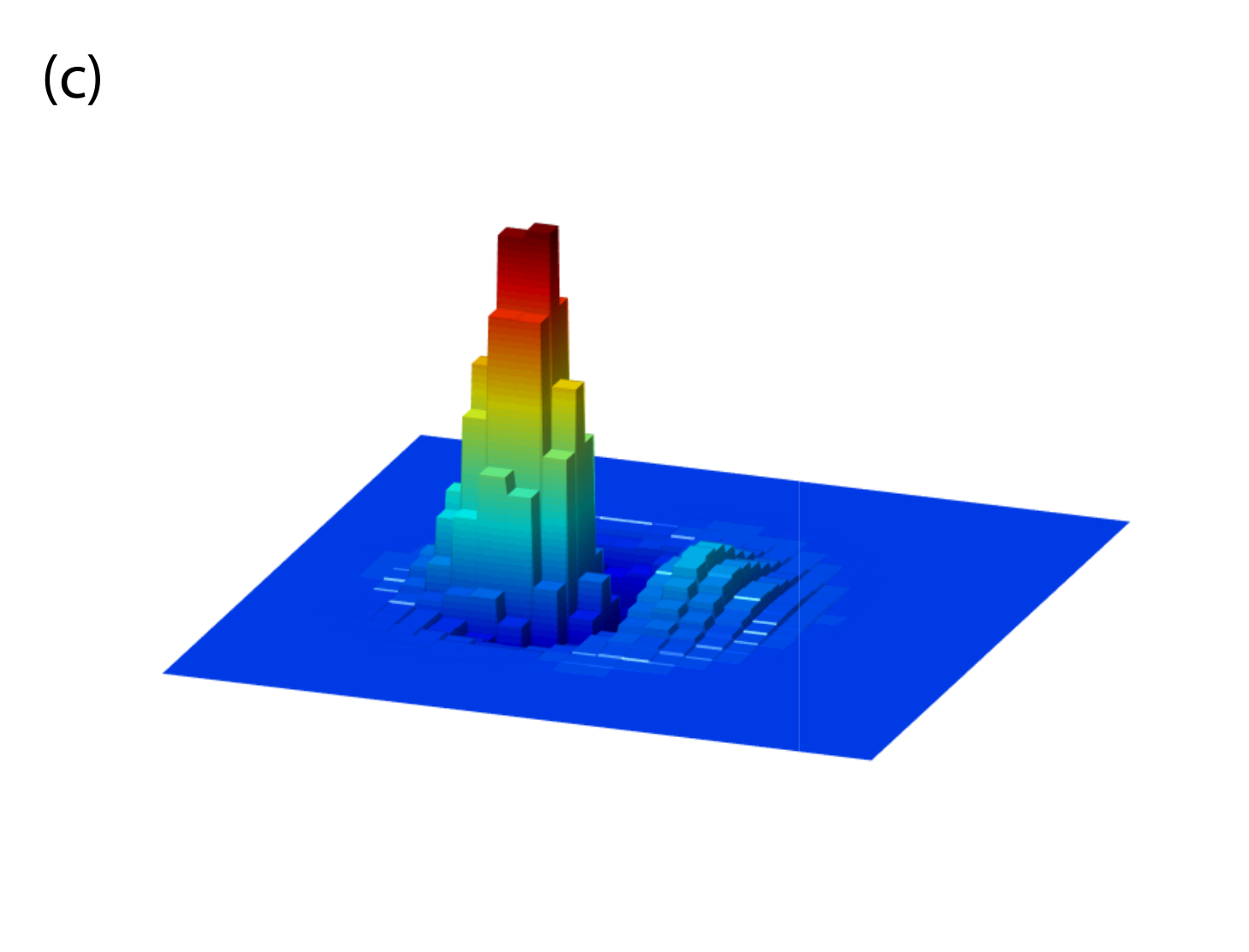}\\
  \includegraphics[height=.18\textheight]{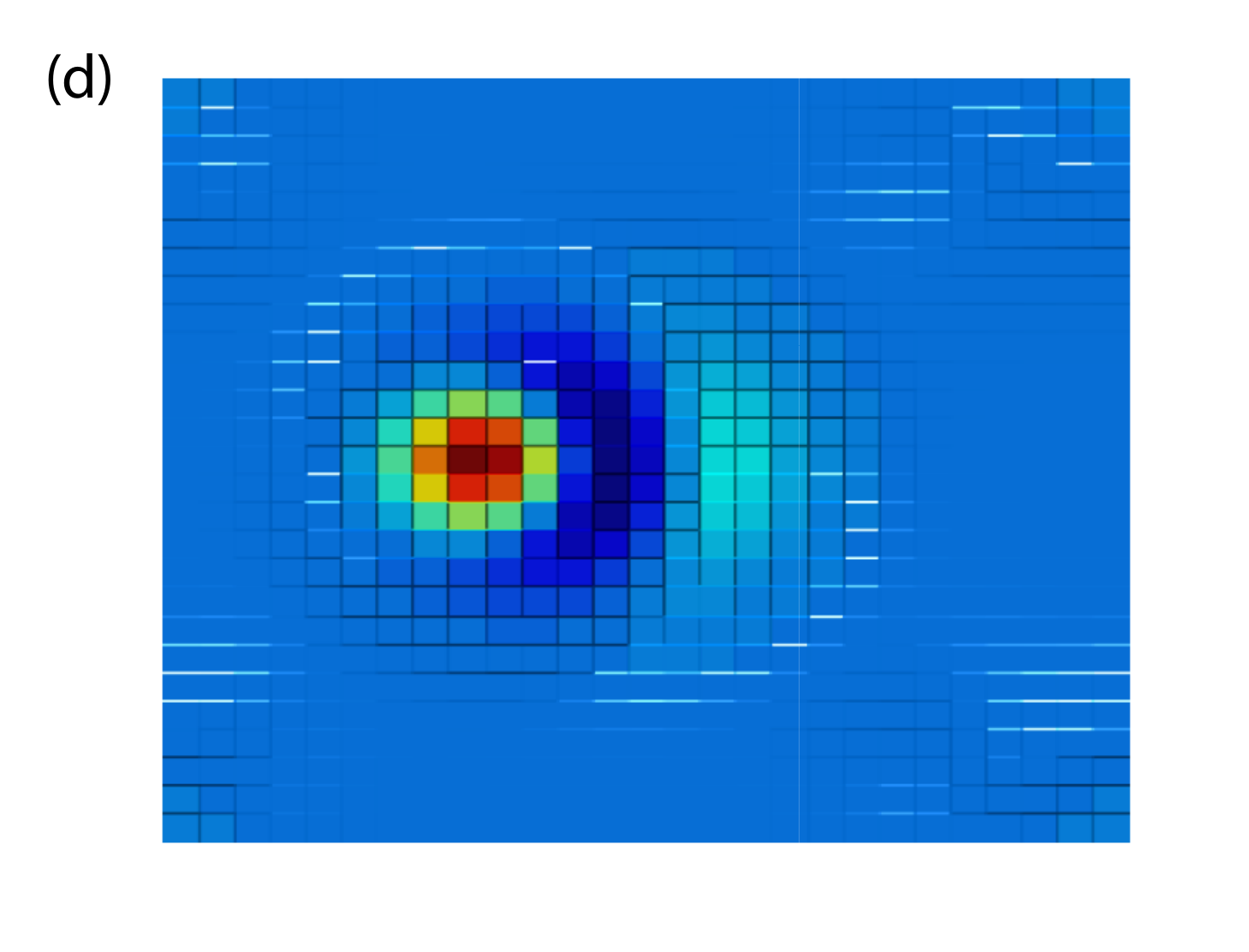}&
  \includegraphics[height=.18\textheight]{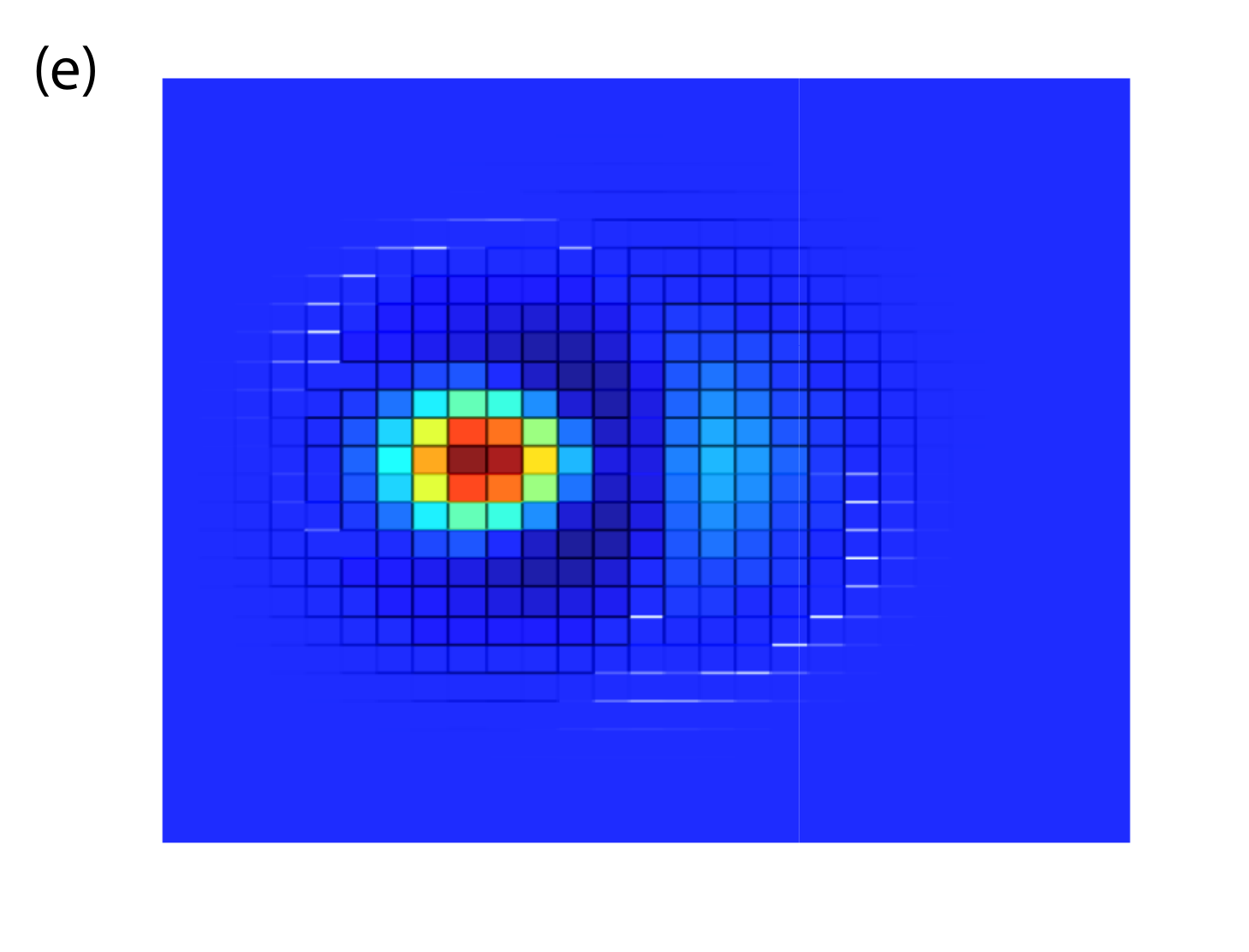}&
  \includegraphics[height=.18\textheight]{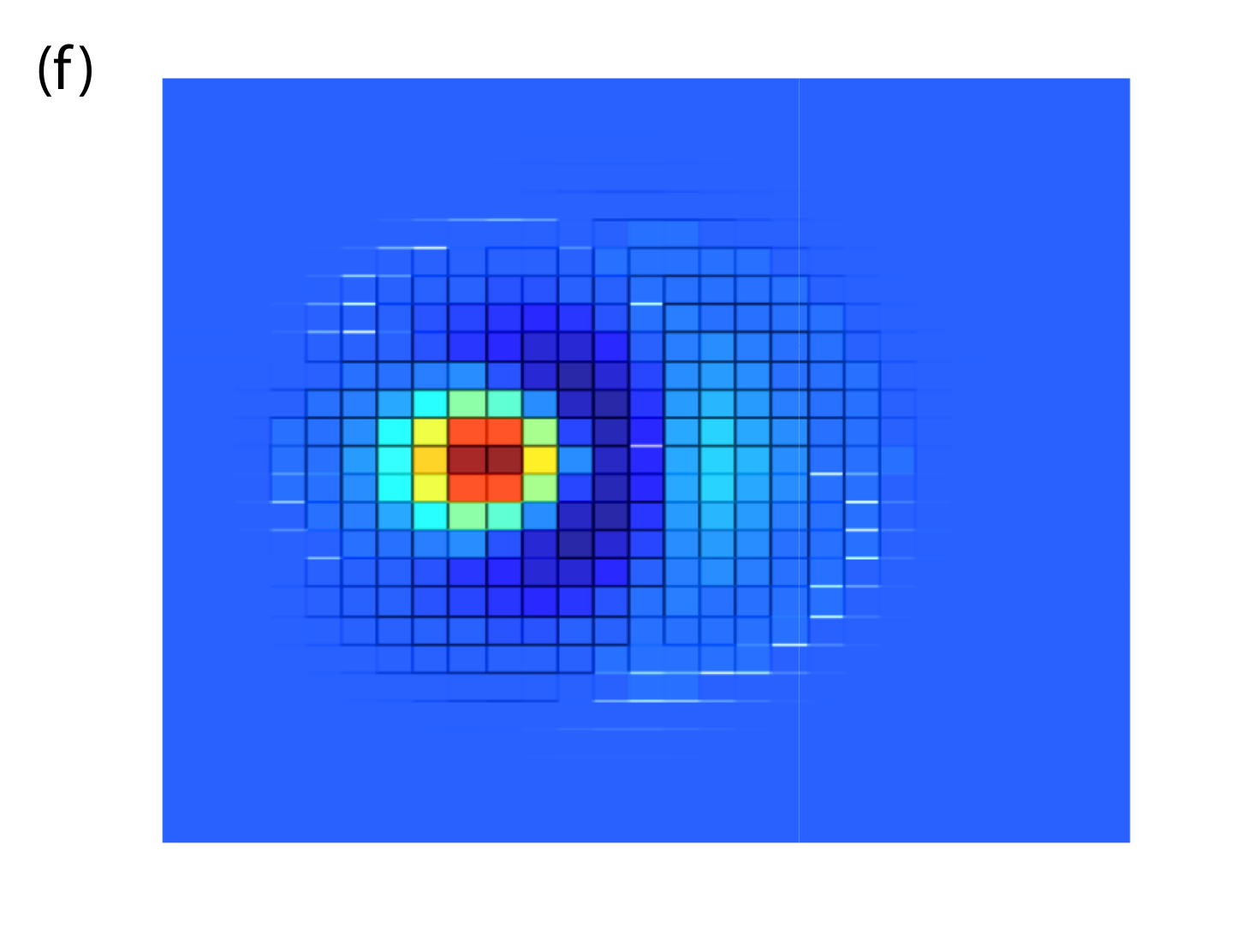}\\
\end{tabular}
\caption{\label{fig01} Evaluation of the collision operator using split and non-split forms: (a) and (d) the split form evaluated using the Fourier transform; (b) and (e) the split form evaluated directly; (c) and (f) the non-split form evaluated using the Fourier transform.}
\end{figure}
\end{center}

Another important issue that makes the non-split formulation more attractive is concerned 
with conservation of mass, momentum, and energy in the discrete solutions. It is the property of 
the exact Boltzmann collision operator that its mass, momentum, and temperature moments are zero. 
Generally, the conservation laws are satisfied only approximately when the Boltzmann equation is 
discretized. Many numerical approaches include mechanisms dedicated to enforcement of  
the conservation laws in discrete solutions in order to guarantee a physically meaningful result.

 \begin{table}[h]
  \begin{tabular}[c]{ c c c c c | c c c c }
  \hline
  \multicolumn{5}{c|}{Error in Conservation of Mass} &
  \multicolumn{4}{c}{Error in Conservation of Temperature}\\
    \hline 
    & \multicolumn{2}{c|}{Split} & \multicolumn{2}{c|}{Non-split} &
    \multicolumn{2}{c|}{Split} & \multicolumn{2}{c}{Non-split} \\
    \hline
    $n$ & Fourier& Direct & Fourier & Direct 
        & Fourier& Direct & Fourier & Direct \\
    \hline
    9   & 0.37 & 1.26 & 1.71E-5 & 1.92E-5 &
          3.51 & 1.69 & 1.71E-2 & 1.84E-2  \\
    15  & 0.10 & 1.20 & 1.45E-5 & 1.71E-5 & 
          0.29 & 1.25 & 1.64E-3 & 3.15E-3 \\
    21  & 0.18 & 1.18 & 0.67E-5 & 0.93E-5 & 
          1.38 & 1.24 & 5.61E-5 & 1.75E-3 \\
    27  & 0.18 & 1.18 & 0.61E-5 & 0.86E-5 & 
          1.37 & 1.24 & 5.40E-4 & 1.05E-3 \\
    \hline
  \end{tabular}
\caption{\label{tab02} Absolute errors in conservation of mass and temperature in the discrete collision integral computed using split and non-split formulations.}
\end{table}

It was observed that if no measures are introduced to enforce the conservation laws,  
solutions to the problem of spatially homogeneous relaxation obtained using the 
split formulation of the collision integral exhibit large, on the order of 5\%\ errors in 
temperature. The mass and momentum are also poorly conserved in this case. 
At the same time, solutions obtained using the non-split formulation had their mass, momentum,
and temperature accurate to three or more digits. To further explore this phenomena, 
we evaluated the collision operator in both split and non-split forms and computed its mass, 
momentum, and temperature moments. The solution was taken to be the sum of two Maxwellians 
in the example above. The numbers of velocity cells were 
varied from 9 to 27. In both split and non-split 
formulations of the collision integral, the decomposed form (\ref{dec01}) of the solution was 
used. For both forms, evaluation of the collision operator was done directly and using the Fourier transform. 
The results are summarized in Table~\ref{tab02}. It can be seen that errors in the mass and 
temperature in the non-split formulation are several orders of magnitude smaller than in 
the split formulation. The errors are also larger in the case of direct evaluation. 
A possible explanation to this is the combined effect of finite precision 
arithmetic and truncation errors in integration that lead to catastrophic cancellation 
when gain and loss terms are combined. 
We note that in 
both split and non-split forms, fulfilment of 
conservation laws requires exact cancellation of the respective integration 
sums. 
When the gain and
loss terms are computed separately using numerical quadratures, the 
relative truncation errors are expected to be acceptable for each of the terms. 
This may change, however, when the terms are combined. It is 
conceivable that significant digits cancel in the two terms and the 
truncation errors are promoted into significance, manifesting in strong 
violations of conservation laws. At the same time, increasing 
the number of velocity cells may not remedy the problem 
due to the expected accumulation of roundoff 
errors. Indeed, evaluation of the gain term in (\ref{dec03}) requires 
$O(M^8)$ arithmetic operations. 
It is possible that combination of large and small values in the finite 
precision arithmetic results in loss of low order digits and a significant 
accumulation of roundoff. When the gain and loss terms are combined, 
this, again, will lead to loss of significance and to perturbations of 
conservation laws. In the case when both the  
non-split form and the decomposition (\ref{dec01}) are used, 
much of the cancellation is happening 
on the level of the integrand. 
We hypothesize here 
that the resulting values of the integrand are smaller and vary less 
in scale. As a result, the accumulated absolute truncation and 
roundoff errors are also smaller, which gives better accuracy in 
conservation laws.

Because of the poor conservation properties and because of the 
susceptibility to aliasing errors we do not recommend the 
split form (\ref{dec03}) for numerical implementation. 

\subsection{Simulations of the spatially homogeneous relaxation}

In this section we present results of solution of the problem of 
spatially homogeneous relaxation using Fourier evaluation of the 
collision operator. Two cases of initial data were considered. In 
both cases, the initial data is a sum of two Maxwellian densities. 
In the first case, the 
dimensionless densities, bulk velocities, and temperatures of the 
Maxwellians are $n_1=1.0007$, $n_2=2.9992$, 
$\bar{\vec{u}}_{1}=(1.2247,0,0)$, $\bar{\vec{u}}_{2}=(0.4082,0,0)$,
$T_{1}=0.2$, $T_{2}=0.7333$. These parameters correspond to upstream 
and downstream conditions of the Mach 3 normal shock wave.  
In the second case, we use the parameters of the example of the previous  
section: $n_{1}=1.6094$, $n_{2}=2.8628$, $\bar{\vec{u}}_{1}=(0.7750,0,0)$, $\bar{\vec{u}}_{2}=(0.4357,0,0)$, $T_{1}=0.3$, and $T_{2}=0.464$. These 
parameters correspond to upstream and downstream conditions of a Mach 1.55 
shock wave. 

In Figures~{\ref{fig02}} and {\ref{fig03}}, relaxation of moments in the 
Mach~3.0 and Mach~1.55 solutions are presented. In the case of Mach~3.0, 
$M=33$ velocity cells were used in each velocity dimension with one 
velocity node on each cell, $s=1$. In the case of Mach~1.55, $M=15$ and $s=1$ 
were used. In the computed solutions, the collision operator was evaluated 
both using the Fourier transform and directly. In the Mach 3.0 instance, the directional 
temperature moments were compared to the moments obtained from a DSMC solution
\cite{Boyd1991411}. 

\begin{center}
\begin{figure}[h]
\centering
  \begin{tabular}{@{}cc@{}}
  \includegraphics[height=.218\textheight]{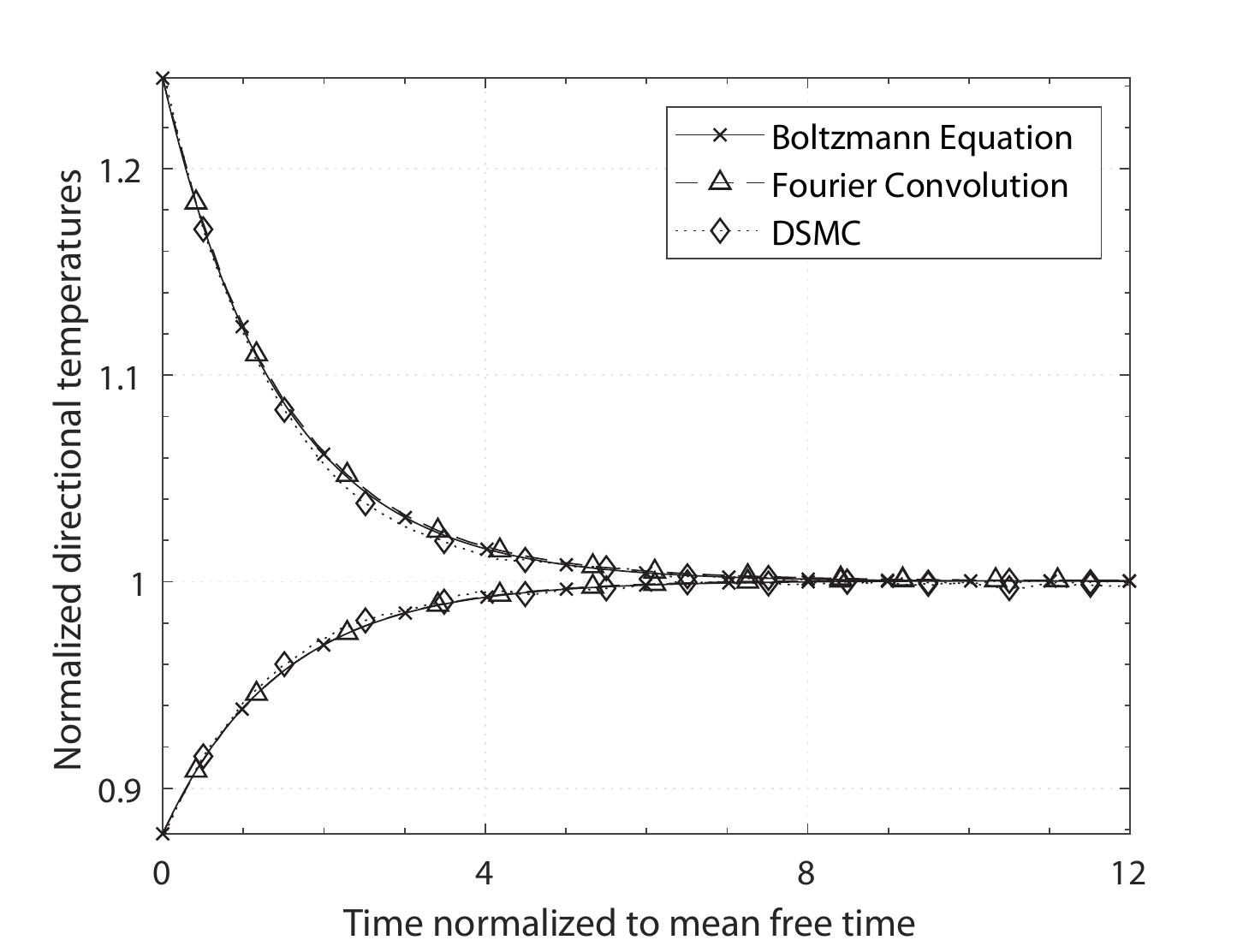}&
  \includegraphics[height=.218\textheight]{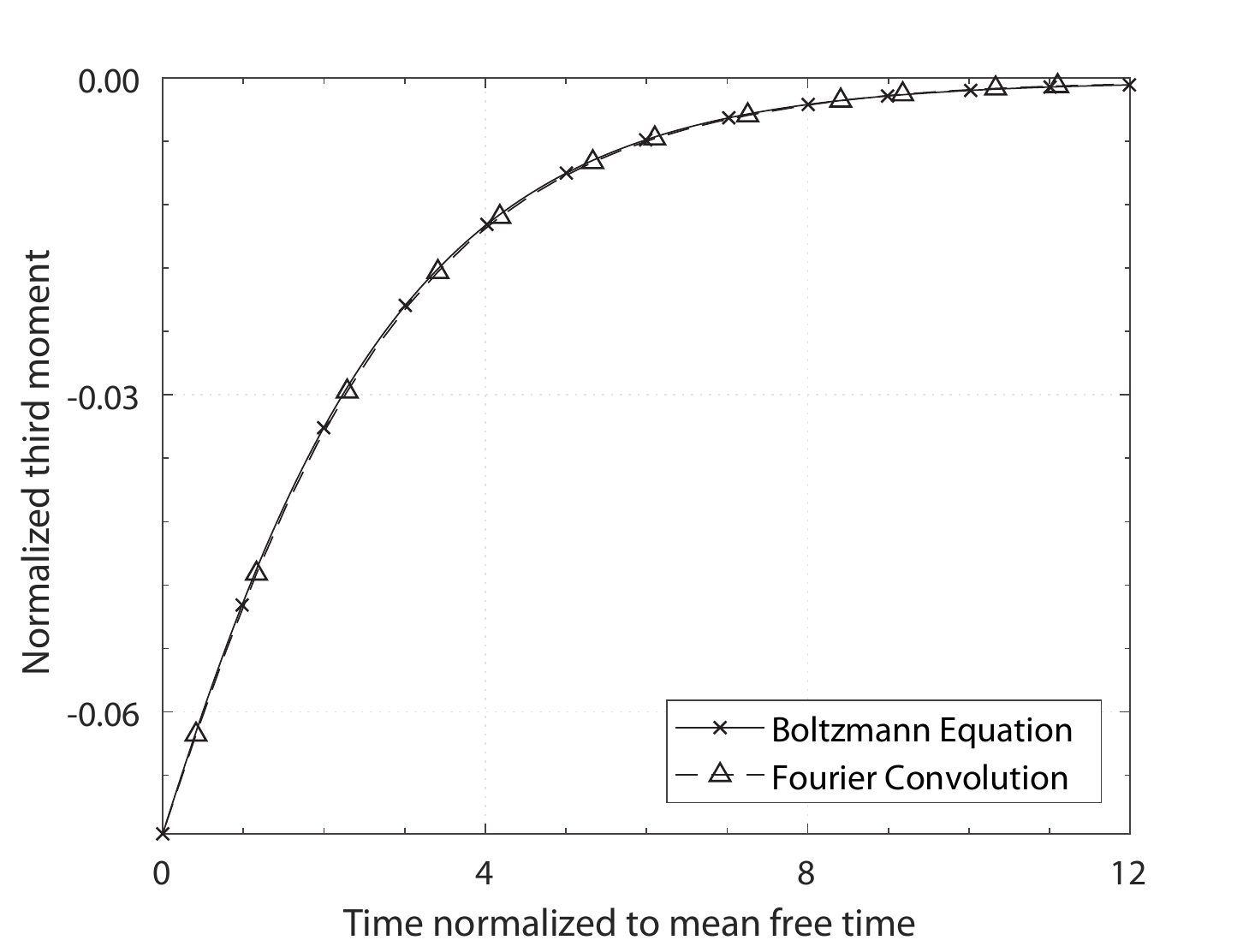}\\      
  \includegraphics[height=.218\textheight]{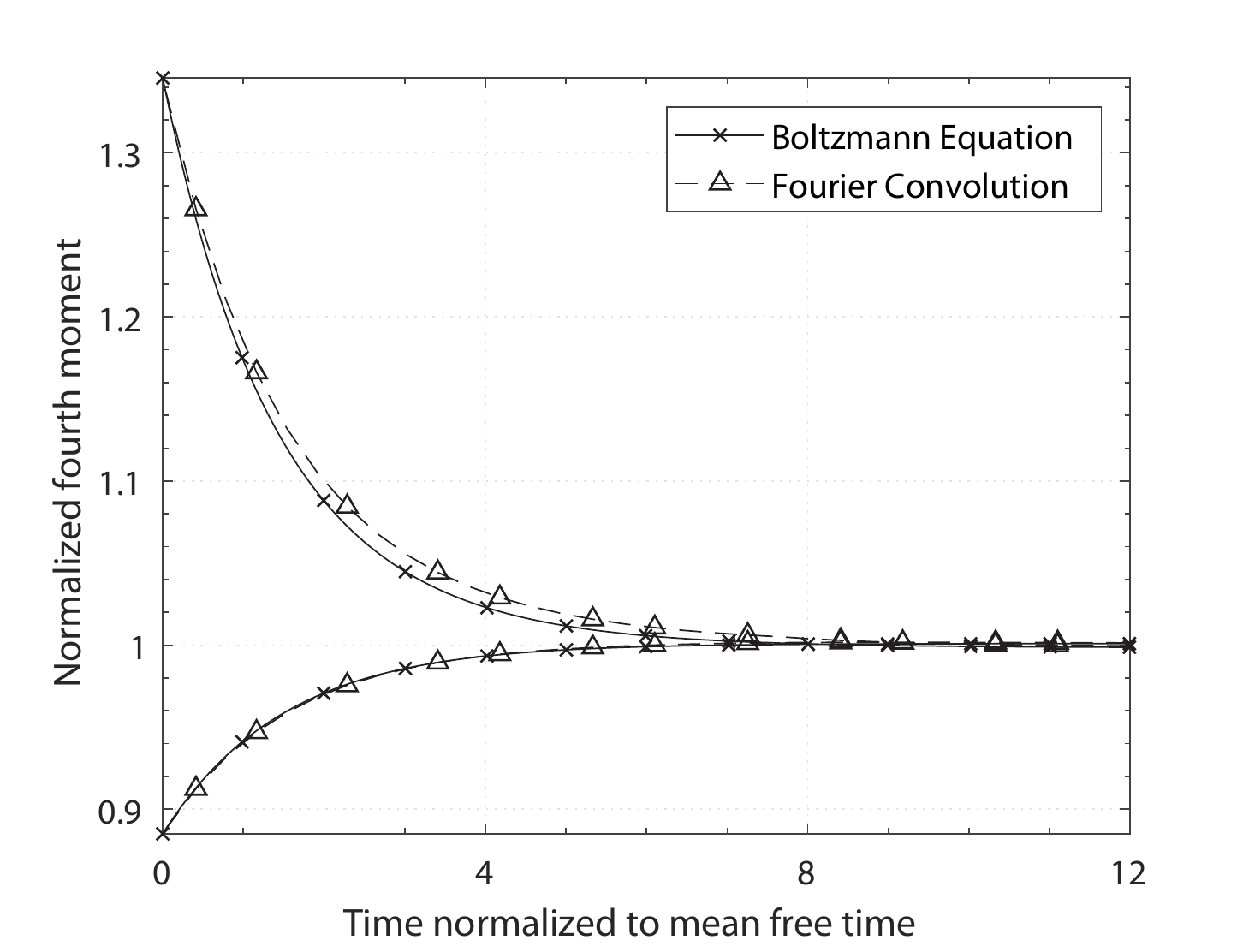}&
  \includegraphics[height=.218\textheight]{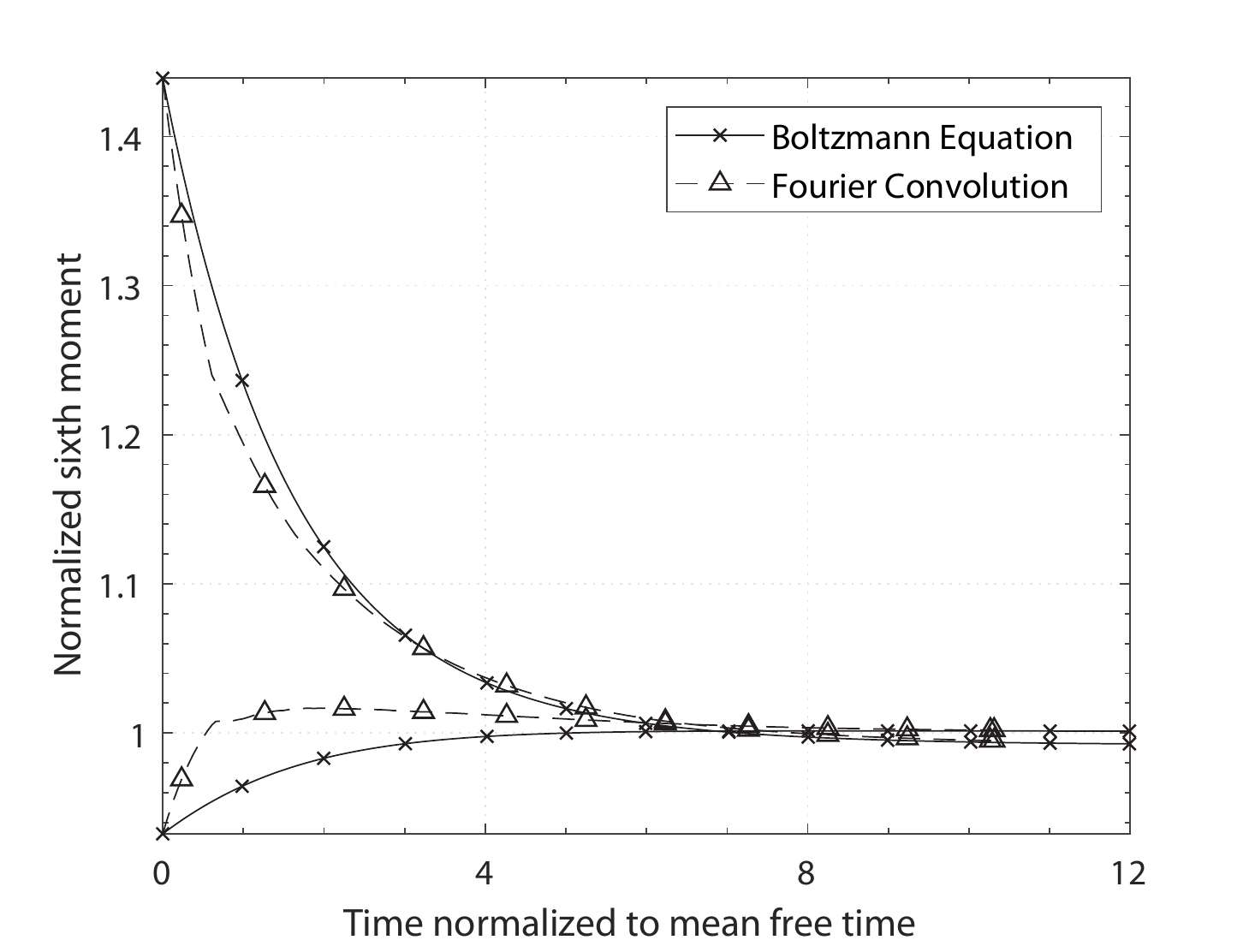}\\
  \end{tabular}
\caption{\label{fig02} Relaxation of moments $f_{\varphi_{i,p}} = 
\int_{R^3} (u_{i}-\bar{u}_{i})^p f(t,\vec{u})\, du$, $i=1,2$, $p=2,3,4,6$ 
in a mix of Maxwellian streams corresponding to a shock wave with 
Mach number 3.0 obtained by solving the Boltzmann equation using Fourier and direct evaluations of the collision integral. In the case of $p=2$, the relaxation of moments 
is also compared to moments of a DSMC solution \cite{Boyd1991411}. }
\end{figure}
\end{center}

\begin{center}
\begin{figure}[h]
\centering
  \begin{tabular}{@{}cc@{}}
  \includegraphics[height=.218\textheight]{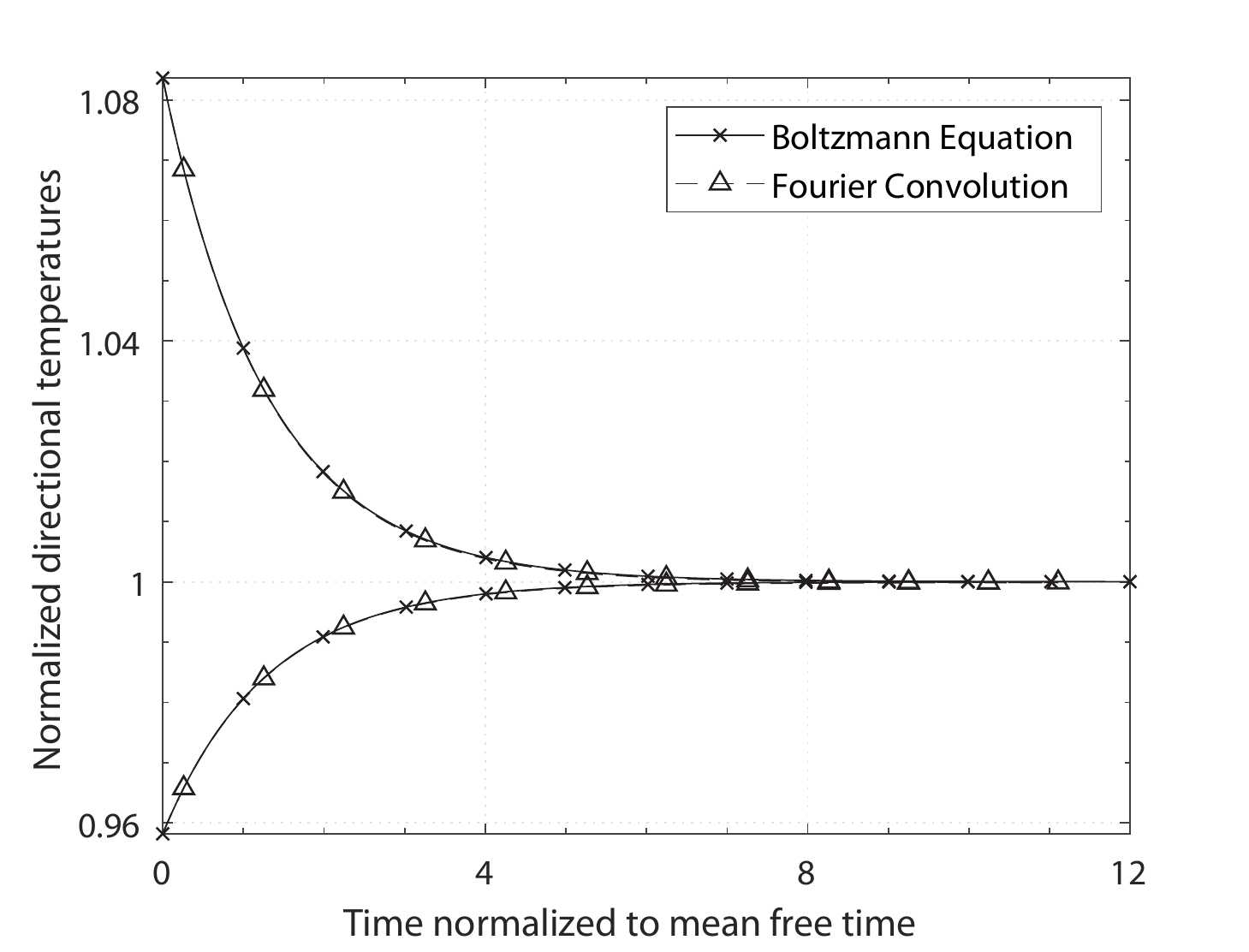}&
  \includegraphics[height=.218\textheight]{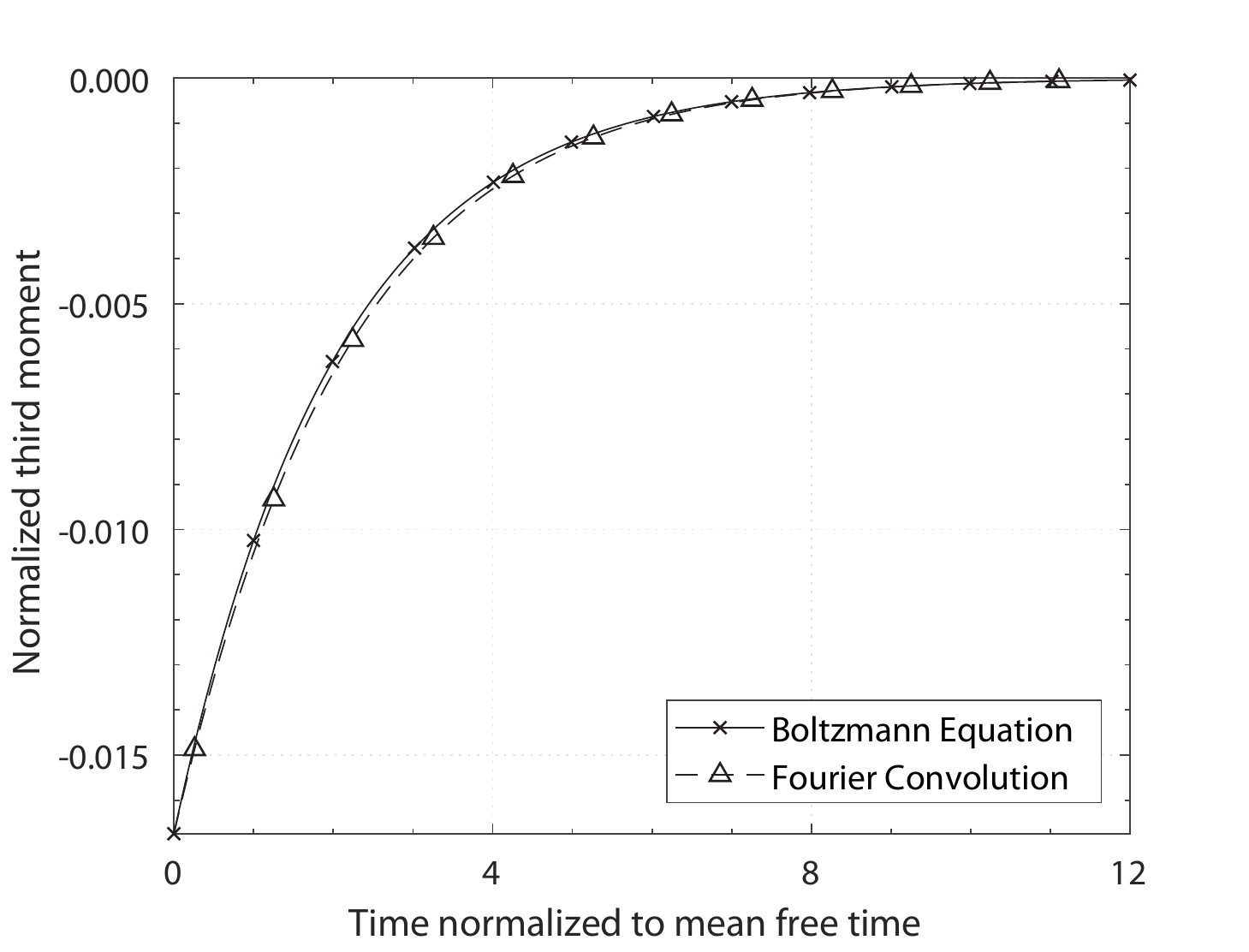}\\      
  \includegraphics[height=.218\textheight]{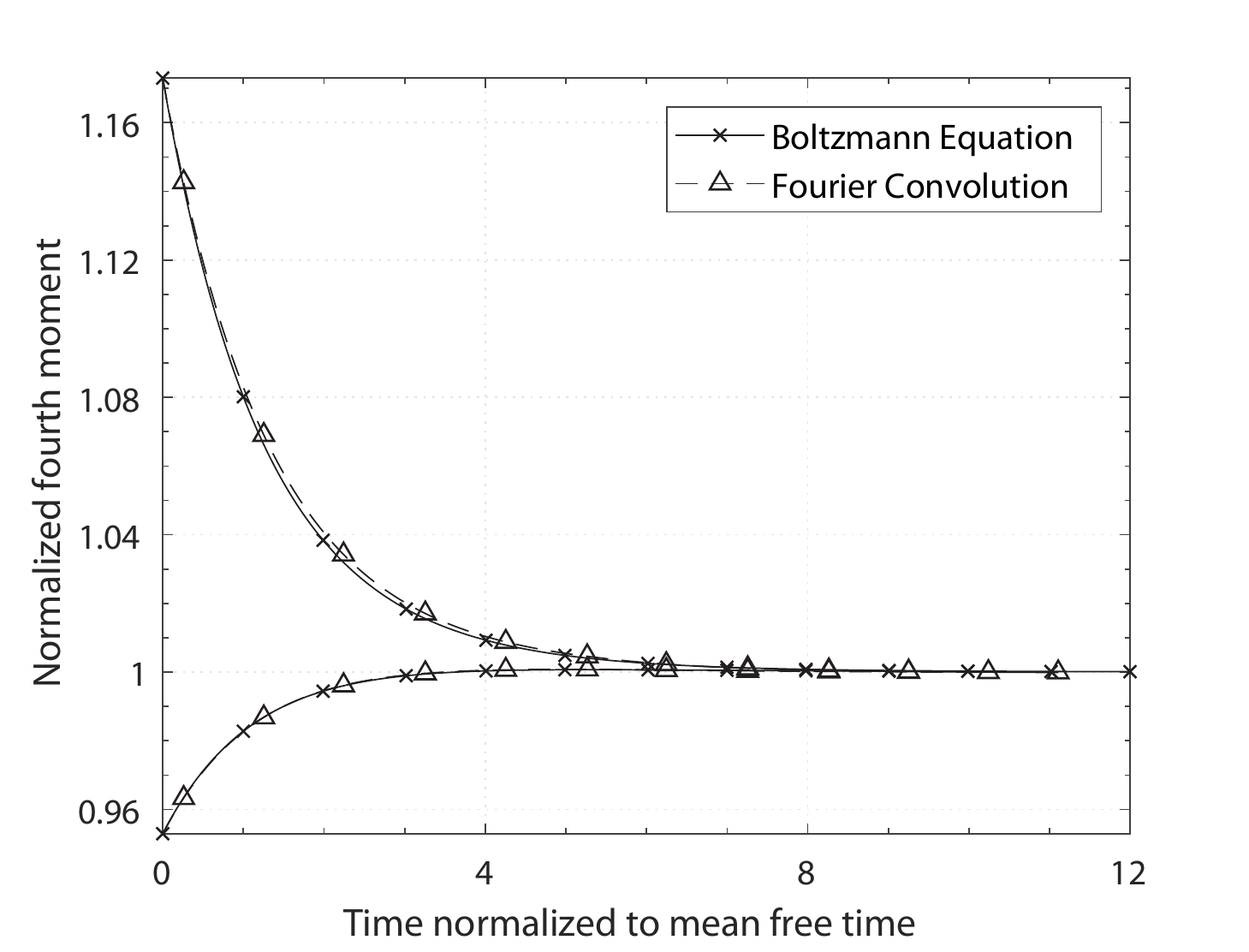}&
  \includegraphics[height=.218\textheight]{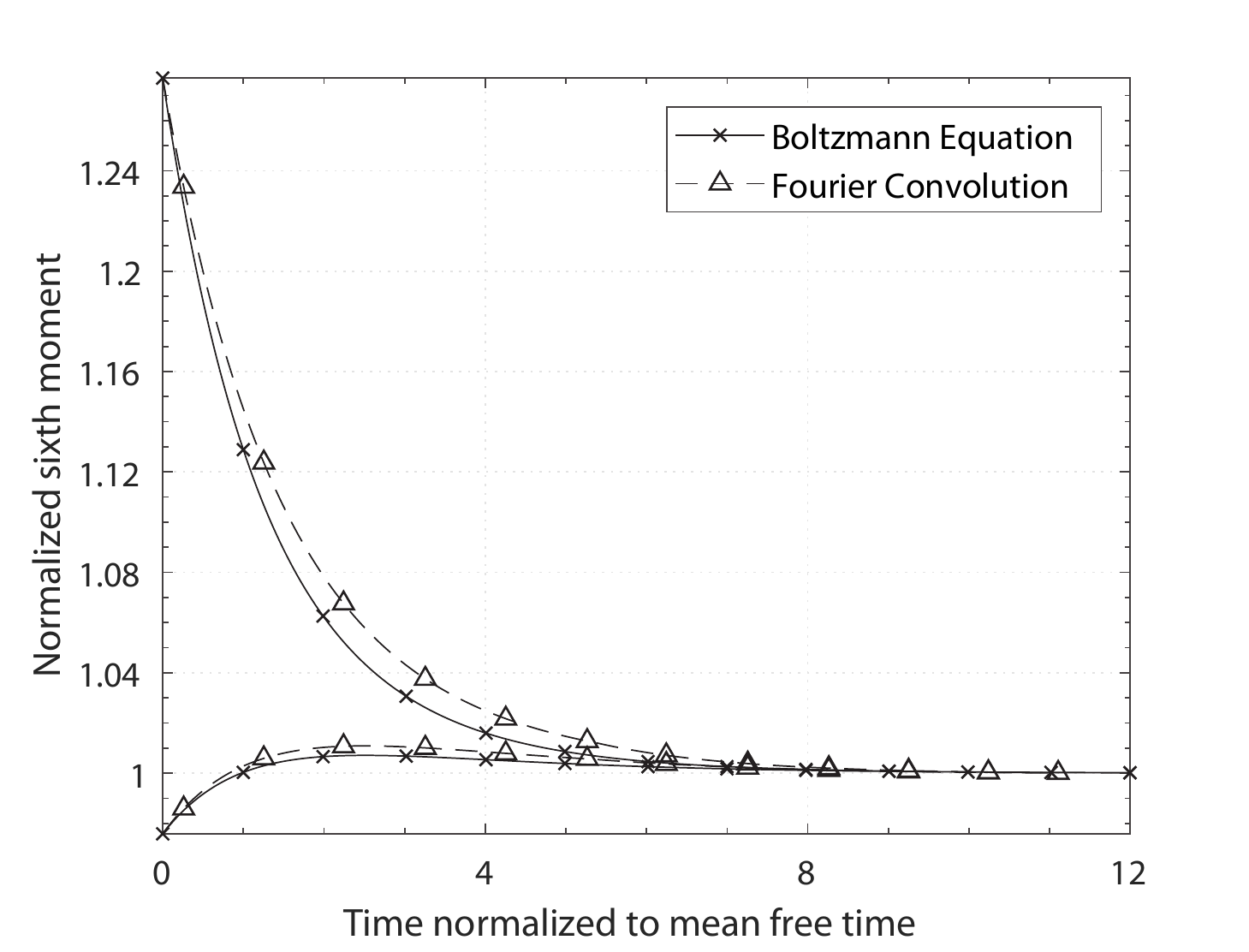}\\
  \end{tabular}
\caption{\label{fig03} Relaxation of moments $f_{\varphi_{i,p}}$, $i=1,2$, $p=2,3,4,6$ 
in a mix of Maxwellian streams corresponding to a shock wave with 
Mach number 1.55 obtained by solving the Boltzmann equation using Fourier and direct evaluations of the collision integral.}
\end{figure}
\end{center}

It can be seen that the solutions obtained by the Fourier evaluation of the collision 
integral are close to those computed by the direct evaluation. The low order moments 
are in excellent agreement for both presented solutions. However, there are 
differences in the higher moments. It appears that the differences are caused 
by a small amount of the aliasing error in the solutions. This can be reduced 
by padding the solution and the kernel with zeros at the expense of higher numerical 
costs, both in time and memory. Overall, however, the $O(M^6)$ evaluation of 
the collision operator using the Fourier transform appears to be consistent 
and stable. 

\section{Conclusion}
We developed and tested an approach for evaluating the Boltzmann collision operator in 
$O(M^6)$ operations where $M$ is the number of velocity cells in one velocity 
dimension. At the basis of the method is the convolution form of the nodal-DG 
discretization of the collision operator \cite{AlekseenkoNguyenWood2015}. The 
algorithm uses the discrete Fourier transform to evaluate convolution fast. 
The method is formulated for uniform grids and for arbitrary order the 
nodal-DG approximation. However, to achieve the $O(M^6)$ complexity, 
it is assumed that the degree of the local DG polynomial basis is kept constant 
and only the numbers of velocity cells are changing. 

The results for the new approach suggest that a potential problem with 
the method could be aliasing errors that perturb higher moments. Aliasing 
errors are results of the assumption of periodicity of the solution 
and the collision kernel. The problem can be remedied by padding solution with zeroes 
as the expense of higher computational time and memory costs which may 
be large. It was observed that the non-split form of the collision operator, 
i.e., when the gain and loss terms are not separated, 
is strongly preferable over the split form for 
the purpose of maintaining conservation laws. Simulations of the problem of 
spatially homogeneous relaxation confirm that the Fourier evaluation of the 
collision operator in this case is accurate and stable for hundreds 
of mean free times. 

While this was not the subject of the present paper, the new method allows for 
scalable MPI parallelization and therefore can potentially be used in multi-dimensional 
problems. Also, generalizations of the analytical convolution form to octree partitions 
of the velocity domain are straightforward. However, difficulties arise when one tries 
to extend fast algorithms for evaluating discrete convolution to octrees. These issues 
will be the authors' future work.

\section*{Acknowledgement}
The authors were supported by the NSF DMS-1620497 grant. 
Computer resources were provided by the Extreme Science 
and Engineering Discovery Environment, 
supported by National Science Foundation Grant No.~OCI-1053575. 
The authors thank professors I.~Gamba and L.~Pareschi for pointing 
out the importance of convolution form in kinetic theory and for inspiring this work. 
The authors thank professors I.~Wood, J.~Shen, and E.~Josyula, and Drs.\ S.~Gimelshein 
and R.~Martin for interest in this work and for fruitful discussions. The authors thank 
R.~Cholvin, B.~Sripimonwan, and E.~Mujica for their interest and help in exploring 
methods for fast evaluation of convolution.

\bibliographystyle{plain}
\bibliography{ffb10092017}

\end{document}